\documentclass[a4paper]{amsart}
\usepackage[UKenglish]{babel}
\usepackage[T2A]{fontenc}

\usepackage{amsmath, amssymb, amscd, amsthm, amsfonts}
\usepackage{mathdots}
\usepackage{mathrsfs}       

\usepackage{float}
\usepackage{graphicx}

\usepackage{bbm} 
\usepackage{microtype} 
\usepackage[normalem]{ulem}  

\usepackage[unicode, pdftex]{hyperref}
\graphicspath{{pictures}}
\usepackage{hyperref}

\usepackage{graphicx}
\usepackage{accents}

\usepackage{bm}

\usepackage[usenames]{color}
\usepackage{colortbl}

\usepackage{colonequals} 
\usepackage{mathtools}

\usepackage[all]{xy}
\usepackage{caption}
\usepackage{subcaption}

\usepackage{tikz,tkz-euclide}
\usepackage{tikz-cd}
\usetikzlibrary{arrows}
\usetikzlibrary {patterns,patterns.meta}
\usepackage{tikz-3dplot}
\usetikzlibrary{shapes.geometric, calc} 

\usepackage{colonequals} 
\usepackage{blkarray, bigstrut} %
\usepackage{comment}

\def\le{\leqslant}
\def\ge{\geqslant}

\DeclareMathOperator{\cone}{cone}

\newcommand{\hr}[2][]{\hyperref[#2]{#1~\ref{#2}}}

\makeatletter
\newtheorem*{rep@theorem}{\rep@title}
\newcommand{\newreptheorem}[2]{%
\newenvironment{rep#1}[1]{%
 \def\rep@title{#2 \ref{##1}}%
 \begin{rep@theorem}}%
 {\end{rep@theorem}}}
\makeatother

\newreptheorem{theorem}{Theorem}

\newtheorem{theorem}{Theorem}[section]
\newtheorem{proposition}[theorem]{Proposition}

\newtheorem{corollary}[theorem]{Corollary}

\theoremstyle{definition}
\newtheorem{example}[theorem]{Example}

\newtheorem{remark}[theorem]{Remark}

\numberwithin{equation}{section}

\DeclareMathAlphabet{\mathbbmsl}{U}{bbm}{m}{sl}

\makeatletter
\@namedef{subjclassname@2020}{%
  \textup{2020} Mathematics Subject Classification}
\makeatother

\begin{document}

\title[Chamber Decompositions of Moment Polytopes for Torus Actions of Positive Complexity]
{Chamber Decompositions of Moment Polytopes \\ for Torus Actions of Positive Complexity}

\author[Matvey A. Sergeev]{Matvey A. Sergeev}
\address{National Research University Higher School of Economics, Russian Federation} \email{matveys.studios@gmail.com}



\begin{abstract}
The present work develops the results of the series of papers by Buchstaber and Terzi\'c on the standard actions of the compact torus $T^n = (S^1)^n$ on the complex Grassmann manifolds $G_{n,2}$. In those works, a hyperplane arrangement in $\mathbb{R}^n$ was introduced that determines the chamber decomposition of the hypersimplex $\Delta_{n,2}$ for the $T^n$-action on $G_{n,2}$. 

We introduce a notion of admissible graph for the standard action of the torus $T^n$ on the complex Grassmannian $G_{n,2}$. In terms of admissible graphs, we give a complete inductive description (with respect to $n \ge 4$) of the admissible polytopes in $\Delta_{n,2}$, as well as of the toric varieties arising as closures of $(\mathbb{C}^*)^n$-orbits on $G_{n,2}$ under the standard $(\mathbb{C}^*)^n$-action. 

We consider the $T^n$-equivariant Pl\"ucker embedding $G_{n,2} \hookrightarrow \mathbb{C}P^{N_2}$, where $N_2 = \binom{n}{2}-1$. Using admissible graphs, for the considered $T^n$-actions, we describe hyperplane arrangements in $\mathbb{R}^n$ that determine the chambers in $\Delta_{n,2}$ for the $T^n$-actions on $G_{n,2}$ and $\mathbb{C}P^{N_2}$. Gel'fand, Kapranov, and Zelevinsky introduced the notions of secondary polytopes and secondary fans in connection with the problem of describing triangulations of a given convex polytope, which is closely related to the Newton polytopes of discriminants and resultants. For the $T^n$-action on $\mathbb{C}P^{N_2}$, we show that the cones in $\mathbb{R}^n$ with vertex at the origin spanned by the chambers form the secondary fan of the cone spanned by the vertices of $\Delta_{n,2}$.
\end{abstract}

\maketitle

\tableofcontents

\section{Introduction}

In toric geometry and toric topology, a model example is the complex projective space $\mathbb{C}P^{n}$ for $n \in \mathbb{N}$ with the canonical actions of the algebraic torus $(\mathbb{C}^*)^{n+1}$ and the induced compact torus $T^{n+1} = (S^{1})^{n+1}$, respectively. It is a nonsingular projective toric variety, and the orbit space of the canonical action of $T^{n+1}$ on $\mathbb{C}P^{n}$ is homeomorphic to the simplex $\Delta^{n}$, which is the moment polytope of this action. A natural generalization of this example is the manifold of projective lines in $\mathbb{C}P^{n}$, i.e., the complex Grassmannian $G_{n,2}$ of $2$-planes in $\mathbb{C}^{n}$, endowed with the standard action of the algebraic torus $(\mathbb{C}^*)^n$ and the induced compact torus $T^{n}$. In this case, however, the moment polytope — the hypersimplex $\Delta_{n,2}$ — is not sufficient to describe the orbit space $G_{n,2}/T^n$. This difficulty arises from the fact that the effective action of $T^{n-1} = T^{n} / \operatorname{diag}(T^{n})$ on $G_{n,2}$ is not half dimensional for $n \ge 4$; rather, it is an action of \emph{positive complexity}, where the complexity is defined as the difference between half the (real) dimension of the manifold and the dimension of the torus that acts effectively. The complexity of the $T^{n-1}$-action on $G_{n,2}$ equals $n-3$, which is positive for $n \ge 4$. The $T^n$-actions on $G_{n,2}$ for $n \ge 4$ served as model examples for the new theory of compact torus actions of positive complexity developed by V.~M.~Buchstaber and S.~Terzi\'c \cite{BT1}.

In 1992, S.~Keel \cite{Keel} described the moduli space $\overline{\mathcal{M}}_{0,n}$ of stable curves of genus $0$ with $n$ marked points via an iterated blow-up procedure, and in 1993, M.~M.~Kapranov \cite{Kapranov} showed that the Chow quotient $G_{n,2} // (\mathbb{C}^*)^n$ of $G_{n,2}$ by the standard action of the algebraic torus $(\mathbb{C}^{*})^n$ is isomorphic to $\overline{\mathcal{M}}_{0,n}$. Aiming to provide a topological description of the orbit space $G_{n,2}/T^n$, Buchstaber and Terzi\'c introduced in 2019 the notion of a \emph{universal space of parameters} $\mathcal{F}_n$ \cite{BT2}, which is a smooth compact manifold. Using the results of Kapranov and Keel, Buchstaber and Terzi\'c constructed in \cite{BT5} an explicit diffeomorphisms between $\overline{\mathcal{M}}_{0,n}$, $G_{n,2} // (\mathbb{C}^*)^n$ and $\mathcal{F}_n$.

An important step for studying the $T^n$-equivariant topology of $G_{n,2}$ is the decomposition into $T^n$-invariant \emph{strata} $W_\sigma$, where the set $\sigma \subseteq \binom{[n]}{2}$, called \emph{admissible}, is a set of index pairs $(i,j)$ with $1 \le i < j \le n$. In terms of the compact torus $T^n$, the strata $W_\sigma$ were introduced by Buchstaber and Terzi\'c \cite{BT2}; they coincide with the strata introduced independently by I.~M.~Gelfand and V.~V.~Serganova \cite{GelfandSerganova} and by M.~Goresky and R.~MacPherson \cite{GelfandGoresky} in terms of the standard action of the algebraic torus $(\mathbb{C}^*)^n$ on $G_{n,2}$. The moment images of the strata $W_\sigma$ are the relative interiors of the so-called \emph{admissible polytopes} $P_\sigma$, which are the convex hulls of the vertices of the hypersimplex $\Delta_{n,2}$ corresponding to the index pairs $(i,j) \in \sigma$.

The admissible polytopes do not form a partition of the hypersimplex, since their interiors may have nonempty intersection in the case of an action of positive complexity. In \cite{GoreskyMacPherson}, Goresky and MacPherson suggested a certain decomposition of the hypersimplex $\Delta_{n,k}$ into disjoint \emph{chambers} $C_{\omega}$, where $\omega$ is a collection of admissible subsets of $[n]$. The chambers are obtained as intersections of the relative interiors of admissible polytopes, where $\omega$ is chosen so that
$$
C_{\omega} := \bigcap_{\sigma \in \omega} \mathring{P}_{\sigma} \quad \text{with } C_{\omega} \neq \varnothing,\quad \text{and}\quad C_{\omega} \cap \mathring{P}_{\sigma} = \varnothing \text{ for all } \sigma \notin \omega.
$$
These chambers yield a decomposition of $\Delta_{n,k}$, which is called a \emph{chamber decomposition}. In the case $k = 2$, Buchstaber and Terzi\'c \cite{BT3} gave an inductive description of the chambers for $\Delta_{n,2} \subset \mathbb{R}^n$ in terms of a hyperplane arrangement $\mathcal{A}_n$ in $\mathbb{R}^n$.

Using the universal space of parameters $\mathcal{F}_n$ and the chamber decomposition of $\Delta_{n,2}$, Buchstaber and Terzi\'c constructed in \cite{BT2,BT3} a \emph{topological model} $(U_n, H_n)$ for the orbit space $G_{n,2}/T^n$, where $U_n = \Delta_{n,2} \times \mathcal{F}_n$ and $H_n \colon \Delta_{n,2} \times \mathcal{F}_n \to G_{n,2}/T^n$ is a continuous projection that in a certain sense resolves the singularities of the orbit space $G_{n,2}/T^n$. An important result of Buchstaber and Terzi\'c is that in order to construct the projection $H_n$, it suffices to do so for $C_{\omega} \times F_{\omega}$, where $F_\omega \subset \mathcal{F}_n$ are the \emph{spaces of parameters of the chambers} $C_{\omega}$, which give a decomposition of $\mathcal{F}_n$.

In \cite{Hasset}, B. Hassett introduced the moduli spaces $\overline{\mathcal{M}}_{0,\mathcal{A}}$ of $\mathcal{A}$-weighted stable curves of genus $0$. In \cite{BT6}, Buchstaber and Terzi\'c constructed a Hassett category whose objects are the spaces $\overline{\mathcal{M}}_{0,\mathcal{A}}$.
Of particular interest is that the structural data of the model $(U_n, H_n)$ can be realized in terms of Hassett category, that is, in terms of spaces $\overline{\mathcal{M}}_{0,\mathcal{A}}$ and the morphisms between them. Moreover, it was proved that the spaces $\overline{\mathcal{M}}_{0,\mathcal{A}}$ admit an embedding into the orbit space $G_{n,2}/T^{n}$.


For the Grassmannian $G_{n,2}$ there is a classical Pl\"ucker embedding into $\mathbb{C}P^{N_2} = \mathbb{C}P(\Lambda^2 \mathbb{C}^n)$, where $N_2 = \binom{n}{2}-1$. The second exterior power $\Lambda^2 \mathbb{C}^n$ of the standard representation of $T^n$ on $\mathbb{C}^n$ gives rise to a homomorphism $T^n \to T^{\binom{n}{2}}$. Composing this with the standard action of $T^{\binom{n}{2}}$ on $\mathbb{C}P^{N_2}$, we obtain a $T^n$-action on $\mathbb{C}P^{N_2}$ of complexity $\binom{n}{2} - n = \frac{n^2 - 3n}{2}$, which is positive for $n \ge 4$. It is essential that the Pl\"ucker embedding is equivariant with respect to the $T^n$-actions on $G_{n,2}$ and $\mathbb{C}P^{N_2}$.

Let $\widehat \mu$ denote the moment map for the standard $T^{\binom{n}{2}}$-action on $\mathbb{C}P^{N_2}$, i.e.
$$\widehat \mu \colon \mathbb{C}P^{N_2} \longrightarrow \big( \mathfrak t^{\binom{n}{2}} \big)^*.$$
The dual Lie algebra homomorphism 
$$\Gamma \colon \big( \mathfrak t^{\binom{n}{2}} \big)^* \longrightarrow (\mathfrak t^n)^*$$
induced by the homomorphism $T^n \to T^{\binom{n}{2}}$ yields a moment map
$$\widetilde \mu \colon \mathbb{C}P^{N_2} \longrightarrow (\mathfrak t^n)^*.$$
By definition, the moment map $\mu$ for the $T^n$-action on $G_{n,2}$ coincides with the composition of $\widetilde \mu$ with the Pl\"ucker embedding. The image of $\widetilde \mu$ is the hypersimplex $\Delta_{n,2}$.

The introduced $T^n$-action on $\mathbb{C}P^{N_2}$ allows us to define a $T^n$-invariant stratification $\widetilde W_\sigma$ of the projective space $\mathbb{C}P^{N_2}$. A stratum $\widetilde W_\sigma$ intersects $G_{n,2}$ if and only if $\sigma$ is an admissible subset of $[n]$, and in that case 
$\widetilde W_\sigma \cap G_{n,2} = W_{\sigma}$. Thus, we can divide the subsets $\sigma \subseteq \binom{[n]}{2}$ into two groups: those for which $\widetilde W_\sigma \cap G_{n,2} = \varnothing$ and those for which $\widetilde W_\sigma \cap G_{n,2} \neq \varnothing$. Subsets $\sigma$ from the first group give rise to convex polytopes $P_{\sigma}$, which are the convex hulls of the vertices of the hypersimplex indexed by the set $\sigma$. This leads to a \emph{new chamber decomposition} of $\Delta_{n,2}$ given by the collection of polytopes $P_\sigma$ spanned by the vertices of $\Delta_{n,2}$ corresponding to $\sigma$ from both groups, i.e., these polytopes are not necessarily admissible. Thus, for each of the considered $T^n$-actions, a chamber decomposition of $\Delta_{n,2}$ is defined. It was shown in \cite{BT0} that these two chamber decompositions of $\Delta_{n,2}$ coincide for $n = 4$ but differ for $n > 4$. However, the chamber decomposition of $\Delta_{n,2}$ corresponding to the $T^n$-action on $\mathbb{C}P^{N_2}$ had not yet been explicitly described.

The present paper has two goals. First, for the standard $T^n$-action on $G_{n,2}$, we introduce the notion of an \emph{admissible graph}. We give a complete classification of admissible graphs as complete multipartite graphs. It is known that all admissible polytopes in $\Delta_{n,2}$ of dimension at most $n-3$ lie on the boundary $\partial \Delta_{n,2}$, and the boundary $\partial \Delta_{n,2}$ itself is the union of $n$ copies of the hypersimplices $\Delta_{n-1,2}$ and $n$ copies of the simplices $\Delta^{n-2}$. Therefore, many results concerning $\Delta_{n,2}$ are inductive in nature. Using the language of admissible graphs, we inductively describe, with respect to $n$, the toric varieties arising as closures of orbits of the standard $(\mathbb{C}^*)^n$-action on $G_{n,2}$ (see Theorems~\ref{(n-2)-dim_adm} and \ref{(n-1)-dim_adm}) as well as the corresponding moment polytopes, which are the admissible polytopes. Furthermore, we describe the chamber decomposition of $\Delta_{n,2}$ arising from the $T^n$-action on $G_{n,2}$, as given by Buchstaber and Terzi\'c in \cite{BT3}, in terms of admissible graphs. Specifically, we establish a one-to-one correspondence between the hyperplane arrangement $\mathcal{A}_n$ and complete bipartite graphs on $n$ vertices.  

Second, we study the chamber decomposition of $\Delta_{n,2}$ with respect to the $T^n$-action on $\mathbb{C}P^{N_2}$. We show that this decomposition is the intersection of the GKZ (Gel'fand--Kapranov--Zelevinsky) fan $\Sigma_{\Gamma}$, or \emph{secondary fan}, of the configuration of vectors in $(\mathfrak t^n)^*$ defined by the projection $\Gamma$ (see \cite{GelfandKapranovZelevinsky} or \cite{Panov, Arzhantsev}), with the hypersimplex $\Delta_{n,2}$. Furthermore, we introduce a new hyperplane arrangement $\widetilde{\mathcal{A}}_n$ in $(\mathfrak t^n)^*$ that determines the chambers of maximal dimension in $\Delta_{n,2}$. As a consequence, we obtain a combinatorial description of the polytopes $P_\sigma$ for arbitrary (not necessarily admissible) $\sigma$, as well as the regular values of the moment map $\widetilde \mu$.

In the final section, we reinterpret our results on the $T^n$-action on $G_{n,2}$ in terms of the complexes of admissible polytopes introduced by Buchstaber and Terzi\'c within the framework of $(2n,k)$-theory \cite{BT1}.


\section{Preliminaries: Torus actions, moment maps, and Gale duality}

\subsection{Torus actions and moment maps}

Let $G_{n,2} = G_2(\mathbb{C}^n)$ denote the Grassmannian of complex $2$-planes in $\mathbb{C}^n$, and let $T^n = (S^1)^n$ be a compact torus acting on $G_{n,2}$ in the standard way. There is a classical \emph{Pl\"ucker embedding}
\begin{equation}
    p \colon G_{n,2} \longrightarrow \mathbb{C}P^{\binom{n}{2}-1} = \mathbb{C}P(\Lambda^2 \mathbb{C}^n),\quad L \mapsto (P^{ij}(L))_{1\le i<j\le n},
\end{equation}
which takes a $2$-plane $L \in G_{n,2}$ to its homogeneous coordinates $P^{ij}(L)$ in $\mathbb{C}P^{N_2} = \mathbb{C}P^{\binom{n}{2}-1}$, called \emph{Pl\"ucker coordinates}. The second exterior power $\Lambda^2 \mathbb{C}^n$ of the standard representation of $T^n$ on $\mathbb{C}^n$ gives rise to a homomorphism $T^n \to T^{\binom{n}{2}}$. Composing this with the standard action of $T^{\binom{n}{2}}$ on $\mathbb{C}P^{N_2}$, we obtain an action of $T^n$ on $\mathbb{C}P^{N_2}$. Observe that the Pl\"ucker embedding is $T^n$-equivariant with respect to the $T^n$-actions on $G_{n,2}$ and $\mathbb{C}P^{N_2}$.

Let $\mathfrak t^n$ and $\mathfrak t^{\binom{n}{2}}$ denote the Lie algebras of the tori $T^n$ and $T^{\binom{n}{2}}$, respectively. There are canonical bases $\bm e_1,\ldots, \bm e_n$ in $\mathfrak t^n$ and $\bm e_{12},\ldots, \bm e_{ij},\ldots, \bm e_{n-1,n}$ in $\mathfrak t^{\binom{n}{2}}$ given by the one-parameter subgroups of the coordinate circles. The differential of the homomorphism $T^n \to T^{\binom{n}{2}}$ gives a linear monomorphism
\begin{equation}
    \Gamma^* \colon \mathfrak{t}^n \longrightarrow \mathfrak t^{\binom{n}{2}},\quad \bm e_i \mapsto \sum_{j \ne i} \bm e_{ij}.
\end{equation}
and a dual projection map
\begin{equation}\label{projectionmap}
    \Gamma \colon \big( \mathfrak t^{\binom{n}{2}} \big)^* \longrightarrow (\mathfrak{t}^n)^*,\quad \bm e^{ij} \mapsto \bm e^i + \bm e^j.
\end{equation}

The standard $T^{\binom{n}{2}}$-action on $\mathbb{C}P^{N_2}$ is Hamiltonian with moment map
$$\widehat \mu \colon \mathbb{C}P^{N_2} \longrightarrow (\mathfrak t^{\binom{n}{2}})^*,\quad [z_{12}: \ldots : z_{ij}: \ldots : z_{n-1,n}]_{1\le i < j \le n} \mapsto \frac{1}{\sum_{i<j}|z_{ij}|^2} \sum_{i<j}|z_{ij}|^2 \bm e^{ij}.$$
Therefore the $T^n$-action on $G_{n,2}$ is Hamiltonian, and the corresponding moment map $\mu$ factors as $\mu = \Gamma \circ \widehat \mu \circ p$, where $p$ is the Pl\"ucker embedding. The explicit formula is
\begin{equation}
    \mu \colon G_{n,2} \longrightarrow (\mathfrak{t}^n)^*,\quad L \mapsto \frac{1}{\sum_{i<j}|P^{ij}(L)|^2} \sum_{i<j} |P^{ij}(L)|^2 (\bm e^i + \bm e^j).
\end{equation}
By the convexity theorem \cite{Atiyah}, the image $\mu(G_{n,2})$ is the convex hull of the images of the $T^n$-fixed points, namely $\bm e^i + \bm e^j$, which is the \emph{hypersimplex} $\Delta_{n,2}$. Therefore, $\Gamma$ takes the vertices $\bm e^{ij}$ of the standard simplex $\Delta^N \subset \big( \mathfrak t^{\binom{n}{2}} \big)^*$ to the vertices $\bm e^i + \bm e^j$ of the hypersimplex $\Delta_{n,2} \subset (\mathfrak t^n)^*$. Observe that $\Gamma$ gives a bijection between the vertex sets of simplex $\Delta^N$ and hypersimplex $\Delta_{n,2}$.

There is a nice $T^n$-invariant stratification on the Grassmannian $G_{n,2}$ due to Gel'fand–Serganova \cite{GelfandSerganova} and Buchstaber–Terzi\'c \cite{BT2}. We give the definition of this stratification following Buchstaber and Terzi\'c. Let $\sigma$ be a set of index pairs $\sigma = \{ (i_1,j_1),\ldots, (i_m,j_m)\}$ with $1 \le i_k < j_k \le n$. Then the set 
$$W_{\sigma} = \{ L \in G_{n,2} \mid P^{ij}(L) \neq 0 \Leftrightarrow (i,j) \in \sigma\}$$
is called a \emph{stratum} if it is nonempty, and $\sigma$ is then called an \emph{admissible set}. One can observe that the moment image of the stratum $W_\sigma$ equals the relative interior of the convex polytope 
$$P_\sigma = \operatorname{conv} \{ \bm e^i + \bm e^j \mid (i,j) \in \sigma \} \subseteq \Delta_{n,2},$$
which is called an \emph{admissible polytope}.

The $T^n$-action on $\mathbb{C}P^{N_2}$ is Hamiltonian, and the corresponding moment map $\widetilde \mu$ factors as $\widetilde \mu = \Gamma \circ \widehat \mu$ with explicit formula
\begin{equation}
   \widetilde \mu \colon \mathbb{C}P^{N_2} \longrightarrow (\mathfrak t^n)^*,\quad [z_{12}: \ldots : z_{ij}: \ldots : z_{n-1,n}]_{1\le i < j \le n} \mapsto \frac{1}{\sum_{i<j}|z_{ij}|^2} \sum_{i<j}|z_{ij}|^2 (\bm e^i + \bm e^j). 
\end{equation}
The fixed points of the $T^n$-action on $\mathbb{C}P^{N_2}$ are precisely the Pl\"ucker coordinates of the standard coordinate $2$-planes in $G_{n,2}$. Therefore, the image of the moment map $\widetilde \mu$ is again $\Delta_{n,2}$.

For the $T^{n}$-action on $\mathbb{C}P^{N_2}$, we have a standard stratification given by
$$
\widetilde{W}_{\sigma} = \{ [z_{12}: \ldots : z_{ij}: \ldots : z_{n-1,n}]_{1\le i < j \le n} \in \mathbb{C}P^{N_2} \mid z_{ij} \neq 0 \Leftrightarrow (i,j) \in \sigma \},
$$
where $\sigma$ is an arbitrary set of index pairs, not necessarily admissible. Consequently, one can define convex polytopes $P_{\sigma}$ spanned by the vertices of the hypersimplex indexed by, generally speaking, any set $\sigma$.

\begin{example}
    Consider $P_\sigma$ for $\sigma = \{ (1,2), (1,3), (3,4)\}$. This polytope is the image of the stratum 
$$\widetilde W_\sigma = \{ (z_{12}:z_{13}:0:0:0:z_{34}) \in \mathbb{C}P^{5} \mid z_{12}, z_{13}, z_{34} \neq 0 \}$$
under $\widetilde \mu$. Observe that $\widetilde W_\sigma \cap G_{4,2} = \varnothing$. Therefore, $\widetilde W_\sigma$ is not formed by a stratum in $G_{4,2}$.
\end{example}

\subsection{Linear Gale duality}

Let $V$ be a real $k$-dimensional vector space and let $\Gamma = \{ \gamma_1,\ldots, \gamma_m\}$ be a \emph{configuration} of vectors in the dual space $V^* = \operatorname{Hom}_{\mathbb{R}}(V;\mathbb{R})$, such that $\operatorname{span} \Gamma = V^*$. Repetitions among $\gamma_1,\ldots, \gamma_m$ are allowed.

The configuration $\Gamma$ defines a linear projection $\Gamma \colon \mathbb{R}^m \to V^*$ which takes the $i$-th standard basis vector $\bm e_i \in \mathbb{R}^m$ to $\gamma_i$. Therefore, we have an exact sequence
\begin{equation}\label{exactseq}
    0 \longrightarrow W \longrightarrow \mathbb{R}^m \xrightarrow{\;\Gamma\;} V^* \longrightarrow 0,
\end{equation}
where $W = \ker \Gamma$ and $\dim_{\mathbb{R}} W = m-k$. Identifying $\mathbb{R}^m$ with its dual $(\mathbb{R}^m)^*$ via the isomorphism that takes the standard basis vector $\bm e_i$ to its dual $\bm e^i$, and applying duality to (\ref{exactseq}), we obtain the dual exact sequence
\begin{equation}\label{dualexactseq}
    0 \longrightarrow V \xrightarrow{\;\Gamma^*\;} (\mathbb{R}^m)^*\cong \mathbb{R}^m \xrightarrow{\;A\;} W^* \longrightarrow 0,
\end{equation}
where $\Gamma^*(\bm v) = \big( \langle \gamma_1, \bm v \rangle, \ldots, \langle \gamma_m, \bm v \rangle \big)$
and $A$ is the restriction map from $(\mathbb{R}^m)^* \cong \mathbb{R}^m$ to $W^*$. Let $\bm a_i := A(\bm e_i)$. Then the linear functional $\bm a_i$ takes $\bm w \in W$ to its $i$-th coordinate in the standard basis. By definition, the \emph{linear Gale transform} takes the vector configuration $\Gamma$ in $V^*$ to the vector configuration $A = \{ \bm a_1, \ldots, \bm a_m \}$ in $W^*$, which is said to be the \emph{Gale dual} of $\Gamma$ and vice versa (see \cite{GelfandKapranovZelevinsky} or \cite{Panov, Arzhantsev}).

The $T^n$-action on $G_{n,2}$ naturally gives rise to the exact sequence
\begin{equation}
    0 \longrightarrow W \longrightarrow \big( \mathfrak t^{\binom{n}{2}} \big)^* \cong \mathbb{R}^{\binom{n}{2}} \xrightarrow{\; \Gamma \;} \mathbb{R}^{n} \cong V^*= (\mathfrak{t}^n)^* \longrightarrow 0.
\end{equation}
Here $V=\mathfrak{t}^n$, $W = \ker \Gamma$ with $\dim_{\mathbb{R}} W = \binom{n}{2} - n$, and the configuration $\gamma_{ij} = \Gamma(\bm e_{ij}) = \bm e^i + \bm e^j$ for $1 \le i < j \le n$ is defined. By duality, we obtain the dual exact sequence
\begin{equation}
    0 \longrightarrow V = \mathfrak{t}^n \cong (\mathbb{R}^n)^* \xrightarrow{\; \Gamma^*\;} \big( \mathbb{R}^{\binom{n}{2}} \big)^* \cong \mathbb{R}^{\binom{n}{2}} \xrightarrow{\; A \;} W^* \longrightarrow 0.
\end{equation}
The Gale dual configuration $\bm a_{ij} = A(\bm e_{ij})$ for $1\le i<j\le n$ is defined.


\section{The Grassmannian $G_{n,2}$ and Admissible Graphs}\label{sectionGrassmannian}

In this section, we reinterpret the chamber decomposition of the hypersimplex $\Delta_{n,2}$ given by admissible polytopes. We not only describe the corresponding hyperplane arrangement $\mathcal{A}_n$ in $(\mathfrak t)^* \cong \mathbb{R}^n$ given in \cite{BT3}, but also relate its description to \emph{complete multipartite graphs}, which we call \emph{admissible graphs}. Using this combinatorial object, we also describe by induction the toric varieties arising as closures of orbits of the standard action of the algebraic torus $(\mathbb{C}^*)^n$ on $G_{n,2}$, as well as the corresponding moment polytopes, i.e., all possible admissible polytopes.

\subsection{The Admissible Graph}

We introduce a new notion of the \emph{admissible graph} $G_{\sigma}$. By definition, it is an abstract graph on $n$ vertices such that $E(G_\sigma) = \sigma$. If $\sigma$ is arbitrary, we obtain an arbitrary graph; if $\sigma$ is an admissible set of index pairs, then $G_\sigma$ is an admissible graph. 

Let us describe the geometric meaning of the admissible graph. Consider the standard $(n-1)$-simplex $\Delta^{n-1}$ in $\mathbb{R}^n$. Scaling it by a factor of $2$ yields
\[
2\Delta^{n-1} = \{(x_1,\ldots, x_n)\in \mathbb{R}^n \mid x_1+\ldots+x_n = 2,\ x_i \ge 0\}.
\]
Note the inclusion $\Delta_{n,2} \subseteq 2\Delta^{n-1}$, where the hypersimplex $\Delta_{n,2}$ is inscribed in the dilated simplex $2\Delta^{n-1}$. We illustrate this inclusion for $n=4$, i.e., $\Delta_{4,2} \subseteq 2\Delta^3$, in Fig.~\ref{fig:embed_hypersimplex}.

\begin{figure}[ht]
    \centering
    \tikzset{every picture/.style={line width=0.75pt}} 

\begin{tikzpicture}[x=0.75pt,y=0.75pt,yscale=-1,xscale=1]

\draw    (132,194.04) -- (323.78,278.6) ;
\draw    (400,155.78) -- (323.78,278.6) ;
\draw  [dash pattern={on 0.84pt off 2.51pt}]  (132,194.04) -- (400,155.78) ;
\draw    (288.13,30.6) -- (132,194.04) ;
\draw    (288.13,30.6) -- (323.78,278.6) ;
\draw    (288.13,30.6) -- (400,155.78) ;
\draw [line width=1.5]    (210.06,112.32) -- (227.89,236.32) ;
\draw [shift={(227.89,236.32)}, rotate = 81.82] [color={rgb, 255:red, 0; green, 0; blue, 0 }  ][fill={rgb, 255:red, 0; green, 0; blue, 0 }  ][line width=1.5]      (0, 0) circle [x radius= 4.36, y radius= 4.36]   ;
\draw [shift={(210.06,112.32)}, rotate = 81.82] [color={rgb, 255:red, 0; green, 0; blue, 0 }  ][fill={rgb, 255:red, 0; green, 0; blue, 0 }  ][line width=1.5]      (0, 0) circle [x radius= 4.36, y radius= 4.36]   ;
\draw [line width=1.5]    (305.95,154.6) -- (227.89,236.32) ;
\draw [shift={(227.89,236.32)}, rotate = 133.69] [color={rgb, 255:red, 0; green, 0; blue, 0 }  ][fill={rgb, 255:red, 0; green, 0; blue, 0 }  ][line width=1.5]      (0, 0) circle [x radius= 4.36, y radius= 4.36]   ;
\draw [shift={(305.95,154.6)}, rotate = 133.69] [color={rgb, 255:red, 0; green, 0; blue, 0 }  ][fill={rgb, 255:red, 0; green, 0; blue, 0 }  ][line width=1.5]      (0, 0) circle [x radius= 4.36, y radius= 4.36]   ;
\draw [line width=1.5]  [dash pattern={on 1.69pt off 2.76pt}]  (361.89,217.19) -- (227.89,236.32) ;
\draw [line width=1.5]    (344.06,93.19) -- (361.89,217.19) ;
\draw [shift={(344.06,93.19)}, rotate = 81.82] [color={rgb, 255:red, 0; green, 0; blue, 0 }  ][fill={rgb, 255:red, 0; green, 0; blue, 0 }  ][line width=1.5]      (0, 0) circle [x radius= 4.36, y radius= 4.36]   ;
\draw [line width=1.5]    (344.06,93.19) -- (305.95,154.6) ;
\draw [line width=1.5]    (361.89,217.19) -- (305.95,154.6) ;
\draw [shift={(305.95,154.6)}, rotate = 228.21] [color={rgb, 255:red, 0; green, 0; blue, 0 }  ][fill={rgb, 255:red, 0; green, 0; blue, 0 }  ][line width=1.5]      (0, 0) circle [x radius= 4.36, y radius= 4.36]   ;
\draw [shift={(361.89,217.19)}, rotate = 228.21] [color={rgb, 255:red, 0; green, 0; blue, 0 }  ][fill={rgb, 255:red, 0; green, 0; blue, 0 }  ][line width=1.5]      (0, 0) circle [x radius= 4.36, y radius= 4.36]   ;
\draw [line width=1.5]    (305.95,154.6) -- (210.06,112.32) ;
\draw [line width=1.5]  [dash pattern={on 1.69pt off 2.76pt}]  (344.06,93.19) -- (210.06,112.32) ;
\draw [line width=1.5]  [dash pattern={on 1.69pt off 2.76pt}]  (266,174.91) -- (210.06,112.32) ;
\draw [shift={(266,174.91)}, rotate = 228.21] [color={rgb, 255:red, 0; green, 0; blue, 0 }  ][fill={rgb, 255:red, 0; green, 0; blue, 0 }  ][line width=1.5]      (0, 0) circle [x radius= 4.36, y radius= 4.36]   ;
\draw [line width=1.5]  [dash pattern={on 1.69pt off 2.76pt}]  (227.89,236.32) -- (266,174.91) ;
\draw [line width=1.5]  [dash pattern={on 1.69pt off 2.76pt}]  (361.89,217.19) -- (266,174.91) ;
\draw [line width=1.5]  [dash pattern={on 1.69pt off 2.76pt}]  (344.06,93.19) -- (266,174.91) ;

\end{tikzpicture}
    \caption{The hypersimplex $\Delta_{4,2}$ (a regular octahedron) inscribed in the dilated tetrahedron $2\Delta^3$.}
    \label{fig:embed_hypersimplex}
\end{figure}
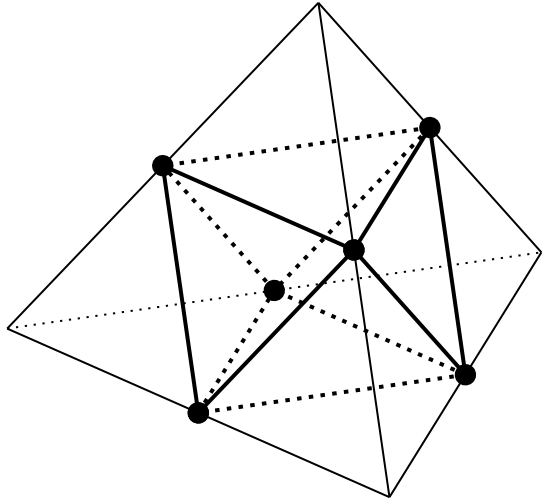

Geometrically, the edges of $\Delta_{n,2}$ realize the \emph{Johnson graph} $J(n,2) = G(\Delta_{n,2})$. 
The embedding $\Delta_{n,2} \hookrightarrow 2\Delta^{n-1}$ reveals that $J(n,2)$ can also be viewed 
as the \emph{line graph} of the edge graph of $2\Delta^{n-1}$. Indeed, the vertices $\bm e_i+ \bm e_j$ of $\Delta_{n,2}$ correspond to the midpoints of the edges of $2\Delta^{n-1}$ (since the vertices of $2\Delta^{n-1}$ are $2\bm e_i$). Therefore, adjacency in $J(n,2)$ corresponds to sharing a vertex in $K_n$. Since the edge graph of $2\Delta^{n-1}$ 
is the complete graph $K_n$, we obtain the well-known identification $J(n,2) = L(K_n)$.

By the embedding $\Delta_{n,2} \hookrightarrow 2\Delta^{n-1}$, the graph $G_{\sigma}$ naturally lives on the $1$-skeleton of $2\Delta^{n-1}$. We have a remarkable duality: the graph $G(P_{\sigma})$ is obtained by taking the line graph of $G_{\sigma}$, which geometrically corresponds to connecting the midpoints of the edges of $2\Delta^{n-1}$ that belong to $G_{\sigma}$.

\begin{example}
    Let $n=4$. Consider the admissible set 
    $$
    \sigma = \bigl\{ (1,2), (2,3), (3,4), (1,4) \bigr\}
    $$
    for the $T^4$-action on $G_{4,2}$. 
    The corresponding admissible polytope $P_{\sigma}$ is a diagonal square in the octahedron $\Delta_{4,2}$.
    The admissible graph $G_{\sigma}$ is the complete bipartite graph $K_{2,2}$ with parts 
    $A_{1}=\{1,3\}$ and $A_{2} = \{2,4\}$; see Fig.~\ref{fig:Admissible_graph_ex}.
\end{example}

\begin{figure}[ht]
    \centering
    \tikzset{every picture/.style={line width=0.75pt}} 

\begin{tikzpicture}[x=0.75pt,y=0.75pt,yscale=-1,xscale=1]

\draw    (132,194.04) -- (323.78,278.6) ;
\draw [color={rgb, 255:red, 0; green, 37; blue, 255 }  ,draw opacity=1 ]   (400,155.78) -- (323.78,278.6) ;
\draw [shift={(323.78,278.6)}, rotate = 121.82] [color={rgb, 255:red, 0; green, 37; blue, 255 }  ,draw opacity=1 ][fill={rgb, 255:red, 0; green, 37; blue, 255 }  ,fill opacity=1 ][line width=0.75]      (0, 0) circle [x radius= 3.35, y radius= 3.35]   ;
\draw [shift={(400,155.78)}, rotate = 121.82] [color={rgb, 255:red, 0; green, 37; blue, 255 }  ,draw opacity=1 ][fill={rgb, 255:red, 0; green, 37; blue, 255 }  ,fill opacity=1 ][line width=0.75]      (0, 0) circle [x radius= 3.35, y radius= 3.35]   ;
\draw [color={rgb, 255:red, 0; green, 37; blue, 255 }  ,draw opacity=1 ] [dash pattern={on 0.84pt off 2.51pt}]  (132,194.04) -- (400,155.78) ;
\draw [color={rgb, 255:red, 0; green, 37; blue, 255 }  ,draw opacity=1 ]   (288.13,30.6) -- (132,194.04) ;
\draw [shift={(132,194.04)}, rotate = 133.69] [color={rgb, 255:red, 0; green, 37; blue, 255 }  ,draw opacity=1 ][fill={rgb, 255:red, 0; green, 37; blue, 255 }  ,fill opacity=1 ][line width=0.75]      (0, 0) circle [x radius= 3.35, y radius= 3.35]   ;
\draw [shift={(288.13,30.6)}, rotate = 133.69] [color={rgb, 255:red, 0; green, 37; blue, 255 }  ,draw opacity=1 ][fill={rgb, 255:red, 0; green, 37; blue, 255 }  ,fill opacity=1 ][line width=0.75]      (0, 0) circle [x radius= 3.35, y radius= 3.35]   ;
\draw [color={rgb, 255:red, 0; green, 37; blue, 255 }  ,draw opacity=1 ]   (288.13,30.6) -- (323.78,278.6) ;
\draw [shift={(323.78,278.6)}, rotate = 81.82] [color={rgb, 255:red, 0; green, 37; blue, 255 }  ,draw opacity=1 ][fill={rgb, 255:red, 0; green, 37; blue, 255 }  ,fill opacity=1 ][line width=0.75]      (0, 0) circle [x radius= 3.35, y radius= 3.35]   ;
\draw [shift={(288.13,30.6)}, rotate = 81.82] [color={rgb, 255:red, 0; green, 37; blue, 255 }  ,draw opacity=1 ][fill={rgb, 255:red, 0; green, 37; blue, 255 }  ,fill opacity=1 ][line width=0.75]      (0, 0) circle [x radius= 3.35, y radius= 3.35]   ;
\draw    (288.13,30.6) -- (400,155.78) ;
\draw [line width=1.5]    (210.06,112.32) -- (227.89,236.32) ;
\draw [shift={(227.89,236.32)}, rotate = 81.82] [color={rgb, 255:red, 0; green, 0; blue, 0 }  ][fill={rgb, 255:red, 0; green, 0; blue, 0 }  ][line width=1.5]      (0, 0) circle [x radius= 4.36, y radius= 4.36]   ;
\draw [shift={(210.06,112.32)}, rotate = 81.82] [color={rgb, 255:red, 0; green, 0; blue, 0 }  ][fill={rgb, 255:red, 0; green, 0; blue, 0 }  ][line width=1.5]      (0, 0) circle [x radius= 4.36, y radius= 4.36]   ;
\draw [line width=1.5]    (305.95,154.6) -- (227.89,236.32) ;
\draw [shift={(227.89,236.32)}, rotate = 133.69] [color={rgb, 255:red, 0; green, 0; blue, 0 }  ][fill={rgb, 255:red, 0; green, 0; blue, 0 }  ][line width=1.5]      (0, 0) circle [x radius= 4.36, y radius= 4.36]   ;
\draw [shift={(305.95,154.6)}, rotate = 133.69] [color={rgb, 255:red, 0; green, 0; blue, 0 }  ][fill={rgb, 255:red, 0; green, 0; blue, 0 }  ][line width=1.5]      (0, 0) circle [x radius= 4.36, y radius= 4.36]   ;
\draw [line width=1.5]  [dash pattern={on 1.69pt off 2.76pt}]  (361.89,217.19) -- (227.89,236.32) ;
\draw [shift={(227.89,236.32)}, rotate = 171.87] [color={rgb, 255:red, 0; green, 0; blue, 0 }  ][fill={rgb, 255:red, 0; green, 0; blue, 0 }  ][line width=1.5]      (0, 0) circle [x radius= 4.36, y radius= 4.36]   ;
\draw [shift={(361.89,217.19)}, rotate = 171.87] [color={rgb, 255:red, 0; green, 0; blue, 0 }  ][fill={rgb, 255:red, 0; green, 0; blue, 0 }  ][line width=1.5]      (0, 0) circle [x radius= 4.36, y radius= 4.36]   ;
\draw [line width=1.5]    (344.06,93.19) -- (361.89,217.19) ;
\draw [shift={(361.89,217.19)}, rotate = 81.82] [color={rgb, 255:red, 0; green, 0; blue, 0 }  ][fill={rgb, 255:red, 0; green, 0; blue, 0 }  ][line width=1.5]      (0, 0) circle [x radius= 4.36, y radius= 4.36]   ;
\draw [shift={(344.06,93.19)}, rotate = 81.82] [color={rgb, 255:red, 0; green, 0; blue, 0 }  ][fill={rgb, 255:red, 0; green, 0; blue, 0 }  ][line width=1.5]      (0, 0) circle [x radius= 4.36, y radius= 4.36]   ;
\draw [line width=1.5]    (344.06,93.19) -- (305.95,154.6) ;
\draw [shift={(305.95,154.6)}, rotate = 121.82] [color={rgb, 255:red, 0; green, 0; blue, 0 }  ][fill={rgb, 255:red, 0; green, 0; blue, 0 }  ][line width=1.5]      (0, 0) circle [x radius= 4.36, y radius= 4.36]   ;
\draw [shift={(344.06,93.19)}, rotate = 121.82] [color={rgb, 255:red, 0; green, 0; blue, 0 }  ][fill={rgb, 255:red, 0; green, 0; blue, 0 }  ][line width=1.5]      (0, 0) circle [x radius= 4.36, y radius= 4.36]   ;
\draw [color={rgb, 255:red, 255; green, 0; blue, 0 }  ,draw opacity=1 ][line width=1.5]    (361.89,217.19) -- (305.95,154.6) ;
\draw [shift={(305.95,154.6)}, rotate = 228.21] [color={rgb, 255:red, 255; green, 0; blue, 0 }  ,draw opacity=1 ][fill={rgb, 255:red, 255; green, 0; blue, 0 }  ,fill opacity=1 ][line width=1.5]      (0, 0) circle [x radius= 4.36, y radius= 4.36]   ;
\draw [shift={(361.89,217.19)}, rotate = 228.21] [color={rgb, 255:red, 255; green, 0; blue, 0 }  ,draw opacity=1 ][fill={rgb, 255:red, 255; green, 0; blue, 0 }  ,fill opacity=1 ][line width=1.5]      (0, 0) circle [x radius= 4.36, y radius= 4.36]   ;
\draw [color={rgb, 255:red, 255; green, 0; blue, 0 }  ,draw opacity=1 ][line width=1.5]    (305.95,154.6) -- (210.06,112.32) ;
\draw [shift={(210.06,112.32)}, rotate = 203.79] [color={rgb, 255:red, 255; green, 0; blue, 0 }  ,draw opacity=1 ][fill={rgb, 255:red, 255; green, 0; blue, 0 }  ,fill opacity=1 ][line width=1.5]      (0, 0) circle [x radius= 4.36, y radius= 4.36]   ;
\draw [shift={(305.95,154.6)}, rotate = 203.79] [color={rgb, 255:red, 255; green, 0; blue, 0 }  ,draw opacity=1 ][fill={rgb, 255:red, 255; green, 0; blue, 0 }  ,fill opacity=1 ][line width=1.5]      (0, 0) circle [x radius= 4.36, y radius= 4.36]   ;
\draw [line width=1.5]  [dash pattern={on 1.69pt off 2.76pt}]  (344.06,93.19) -- (210.06,112.32) ;
\draw [color={rgb, 255:red, 255; green, 0; blue, 0 }  ,draw opacity=1 ][line width=1.5]  [dash pattern={on 1.69pt off 2.76pt}]  (266,174.91) -- (210.06,112.32) ;
\draw [shift={(266,174.91)}, rotate = 228.21] [color={rgb, 255:red, 255; green, 0; blue, 0 }  ,draw opacity=1 ][fill={rgb, 255:red, 255; green, 0; blue, 0 }  ,fill opacity=1 ][line width=1.5]      (0, 0) circle [x radius= 4.36, y radius= 4.36]   ;
\draw [line width=1.5]  [dash pattern={on 1.69pt off 2.76pt}]  (227.89,236.32) -- (266,174.91) ;
\draw [color={rgb, 255:red, 255; green, 0; blue, 0 }  ,draw opacity=1 ][line width=1.5]  [dash pattern={on 1.69pt off 2.76pt}]  (361.89,217.19) -- (266,174.91) ;
\draw [shift={(266,174.91)}, rotate = 203.79] [color={rgb, 255:red, 255; green, 0; blue, 0 }  ,draw opacity=1 ][fill={rgb, 255:red, 255; green, 0; blue, 0 }  ,fill opacity=1 ][line width=1.5]      (0, 0) circle [x radius= 4.36, y radius= 4.36]   ;
\draw [shift={(361.89,217.19)}, rotate = 203.79] [color={rgb, 255:red, 255; green, 0; blue, 0 }  ,draw opacity=1 ][fill={rgb, 255:red, 255; green, 0; blue, 0 }  ,fill opacity=1 ][line width=1.5]      (0, 0) circle [x radius= 4.36, y radius= 4.36]   ;
\draw [line width=1.5]  [dash pattern={on 1.69pt off 2.76pt}]  (344.06,93.19) -- (266,174.91) ;

\draw (282,11) node [anchor=north west][inner sep=0.75pt]    {$1$};
\draw (321,284) node [anchor=north west][inner sep=0.75pt]    {$2$};
\draw (404.78,147) node [anchor=north west][inner sep=0.75pt]    {$3$};
\draw (115,189) node [anchor=north west][inner sep=0.75pt]    {$4$};
\draw (315.78,144) node [anchor=north west][inner sep=0.75pt]    {$12$};
\draw (345.78,73) node [anchor=north west][inner sep=0.75pt]    {$13$};
\draw (369.78,212) node [anchor=north west][inner sep=0.75pt]    {$23$};
\draw (211.78,245) node [anchor=north west][inner sep=0.75pt]    {$24$};
\draw (194.78,91) node [anchor=north west][inner sep=0.75pt]    {$14$};
\draw (230,161) node [anchor=north west][inner sep=0.75pt]    {$34$};

\end{tikzpicture}
    \caption{The admissible graph $G_{\sigma} = K_{2,2}$ for $\sigma = \{(1,2),(2,3),(3,4),(1,4)\}$.}
    \label{fig:Admissible_graph_ex}
\end{figure}
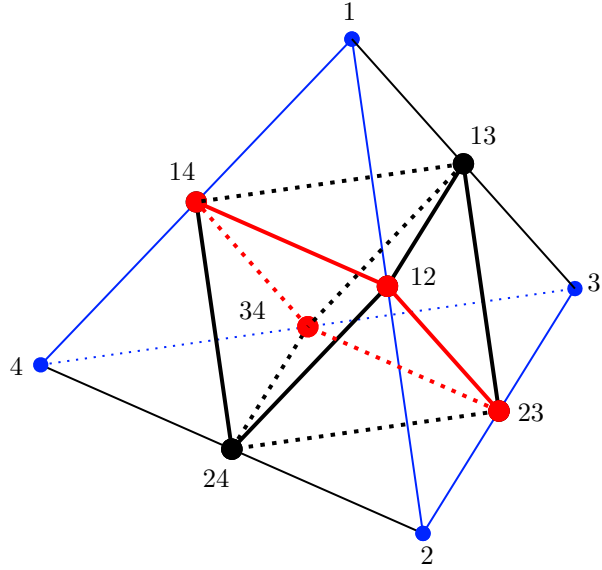

\begin{remark}
Geometrically, the duality $G_{\sigma} \mapsto G(P_{\sigma}) = L(G_{\sigma})$ shares a key feature with \emph{medial graphs}: vertices of the resulting graph correspond to edges of the original, and adjacency reflects shared endpoints. However, our setting is not restricted to planar embeddings and works for arbitrary spanning subgraphs of the complete graph, viewed on the $1$-skeleton of a simplex.
\end{remark}

\begin{remark}
    For general $k \ge 2$, consider the embedding $\Delta_{n,k} \hookrightarrow k\Delta^{n-1}$. 
    In this embedding, the vertex $\sum_{i\in I} \bm e_i$ lies at the centroid 
    $\frac{1}{k}\sum_{i\in I} k \bm e_i$ of the $(k-1)$-dimensional face of $k\Delta^{n-1}$ spanned by 
    the vertices $\{k \bm e_i\}_{i\in I}$. This leads to the following correspondence:
    $$
    \{\text{subgraphs of } J(n,k)\} \longleftrightarrow \{(k-1)\text{-chains in } k\Delta^{n-1}\}.
    $$
\end{remark}

Now we present a complete combinatorial classification of admissible graphs for the standard $T^n$-action on $G_{n,2}$. Since an admissible graph $G_{\sigma}$ is uniquely determined by its corresponding admissible set $\sigma \subseteq \binom{[n]}{2}$, this classification also yields a complete description of all admissible sets.

\begin{theorem}[Combinatorial classification of admissible graphs and sets]\label{classification}
A subset $\sigma \subseteq \binom{[n]}{2}$ is admissible for the $T^n$-action on $G_{n,2}$ if and only if there exists a partition
$$
[n] = A_1 \sqcup \cdots \sqcup A_N \sqcup B \qquad (N \ge 2,\ A_i \neq \varnothing)
$$
such that $\sigma$ consists precisely of all pairs $(i,j)$ with $i$ and $j$ belonging to different parts $A_p, A_q$ ($p \neq q$), and conversely, for every $i \in A_p$ and $j \in A_q$ we have $(i,j) \in \sigma$.

Equivalently, the corresponding admissible graph $G_{\sigma}$ is the disjoint union of the set $B$ of isolated vertices and the complete $N$-partite graph $K(A_1,\dots,A_N)$.
\end{theorem}
\begin{proof}
Let $\sigma \subseteq \binom{[n]}{2}$ be an admissible set for the $T^n$-action on $G_{n,2}$. 
Consider a $2$-plane $L \in W_{\sigma}$, represented by an $n \times 2$ matrix $X$ modulo right $GL(2,\mathbb{C})$ action. 
Interpreting its rows as vectors $v_{1},\ldots, v_{n} \in \mathbb{C}^2$, we have $(i,j) \in \sigma$ 
if and only if $v_i$ and $v_j$ are non-proportional non-zero vectors. The only global constraint on $X$ is that 
the vectors $v_1,\ldots,v_n$ must span $\mathbb{C}^2$, which is equivalent to requiring that at least two of them are non-zero and non-proportional. 

Some of the vectors $v_{1},\ldots, v_{n}$ may be zero; the corresponding indices then correspond to isolated vertices in the admissible graph $G_{\sigma}$. For the non-zero vectors we obtain a partition $A_{1}\sqcup \ldots \sqcup A_{N}$, where $i$ and $j$ belong to the same part if and only if the vectors $v_{i}$ and $v_{j}$ are proportional to each other. Let $B \subset [n]$ be the set of indices corresponding to zero vectors. 

Therefore, $[n] = A_{1}\sqcup \ldots \sqcup A_{N} \sqcup B$, and $G_{\sigma}$ is a disjoint union of $|B|$ isolated vertices together with a complete $N$-partite graph $K(A_1, \ldots, A_N)$.
\end{proof}

Now we compute the dimension of an admissible polytope $P_\sigma$ in terms of the corresponding admissible graph $G_\sigma = \big( \sqcup_{i \in B} \{i\} \big) \sqcup K(A_1,\ldots, A_N)$.

\begin{theorem}\label{dimensionadmissiblepolytope}
    The dimension of the admissible polytope $P_{\sigma}$ is given by
    \begin{equation}
        \dim P_{\sigma} = 
    \begin{cases}
    n - |B| - 2, & \text{if } N = 2,\\[4pt]
    n - |B| - 1, & \text{if } N \ge 3.
    \end{cases}
    \end{equation}
\end{theorem}

\begin{proof}
According to Gel'fand and Serganova \cite{GelfandSerganova}, admissible polytopes are the moment polytopes of (possibly singular) toric varieties arising as closures $\overline{\mathcal{O}_{\mathbb{C}}(L)}$ of $(\mathbb{C}^*)^n$-orbits $\mathcal{O}_{\mathbb{C}}(L)$ under the standard action of the algebraic torus $(\mathbb{C}^*)^n$ on $G_{n,2}$. Thus, for $P_\sigma$ there exists an $L \in G_{n,2}$ such that $P_\sigma = \mu\big{(} \overline{\mathcal{O}_{\mathbb{C}}(L)} \big{)}$. Consequently,
$$
\dim P_\sigma = \dim_{\mathbb{C}}\overline{\mathcal{O}_{\mathbb{C}}(L)} = \dim_{\mathbb{C}} \mathcal{O}_{\mathbb{C}}(L).
$$

As is well known, the dimension of the orbit $\mathcal{O}_{\mathbb{C}}(L)$ equals the codimension of the stabilizer $\operatorname{Stab}(L)$ in $(\mathbb{C}^*)^n$. 
Let $t = (t_1,\ldots, t_n) \in (\mathbb{C}^*)^n$ and let $L$ be represented by an $n \times 2$ matrix $X$ with rows $v_1,\ldots, v_n \in \mathbb{C}^2$, 
where at least two of them are nonzero and non‑proportional. 
Recall that the Pl\"ucker coordinate $P^{ij}(L)$ equals the $2 \times 2$ minor formed by rows $i$ and $j$ of $X$, i.e., $\det(v_i, v_j)$. 
Under the torus action, $t \cdot L$ corresponds to multiplying each row $v_i$ by $t_i$, hence
$$
P^{ij}(t \cdot L) = t_i t_j P^{ij}(L).
$$
Therefore, $t \in \operatorname{Stab}(L)$ if and only if  
$$
t_i t_j = t_{i_0} t_{j_0} \quad \text{for all } (i,j) \in \sigma,
$$
where $(i_0,j_0)$ is any fixed pair in $\sigma$, since all non‑zero Pl\"ucker coordinates of $L$ must be preserved up to a common scalar factor.

As we know from Theorem~\ref{classification}, $G_{\sigma}$ consists of isolated vertices $B$ together with a complete $N$-partite graph $K(A_1,\dots,A_N)$ for some $N \ge 2$. 
The isolated vertices do not affect the stabilizer, as they are incident to no edges. Now consider the $N$-partite connected component $K(A_1,\dots,A_N)$. Observe that if edges $(i_1,i_2)$ and $(i_2,i_3)$ share a vertex $i_2$, then from $t_{i_1} t_{i_2} = t_{i_2} t_{i_3}$ we obtain $t_{i_1} = t_{i_3}$. 
Let $A$ and $A'$ be two distinct parts. For any vertices $i_1, i_2 \in A$ and $i' \in A'$, the path $i_1, i', i_2$ yields $t_{i_1} = t_{i_2}$. 
Consequently, all vertices in a given part $A_j$ must have the same $t$-value; denote it by $\tau_j$ ($j = 1,\dots,N$). 

We now distinguish two cases:

\begin{itemize}
    \item \emph{Case $N = 2$:} There are two independent values $\tau_1$ and $\tau_2$. 
    The isolated vertices contribute $|B|$ free parameters. Hence 
    $$
    \dim_{\mathbb{C}} \operatorname{Stab}(L) = |B| + 2,
    $$
    and therefore
    $$
    \dim_{\mathbb{C}} \mathcal{O}_{\mathbb{C}}(L) = n - \dim_{\mathbb{C}} \operatorname{Stab}(L) = n - |B| - 2.
    $$

    \item \emph{Case $N \ge 3$:} For any three distinct parts, take vertices $i_1 \in A_a$, $i_2 \in A_b$, $i_3 \in A_c$. Then the cycle $i_1, i_2, i_3$ (using edges $\{i_1,i_2\}$, $\{i_2,i_3\}$, $\{i_3,i_1\}$) gives $\tau_a = \tau_b = \tau_c$. 
    Hence all $\tau_j$ coincide. 
    Thus the stabilizer is determined by this common value together with the $|B|$ free parameters from the isolated vertices, yielding
    $$
    \dim_{\mathbb{C}} \operatorname{Stab}(L) = |B| + 1,
    $$
    and consequently
    $$
    \dim_{\mathbb{C}} \mathcal{O}_{\mathbb{C}}(L) = n - |B| - 1.
    $$
\end{itemize}
\end{proof}

\subsection{Chamber decomposition of $\Delta_{n,2}$ for $T^n$-action on $G_{n,2}$}

By definition the \emph{chambers} $C_{\omega}$ of the hypersimplex $\Delta_{n,2}$, where $\omega$ is a collection of admissible sets, are obtained as intersections of the relative interiors of admissible polytopes, namely 
$$
C_{\omega} := \bigcap_{\sigma \in \omega} \mathring{P}_{\sigma} \quad \text{with } C_{\omega} \neq \varnothing,\quad \text{and}\quad C_{\omega} \cap \mathring{P}_{\sigma} = \varnothing \text{ for all } \sigma \notin \omega.
$$

We start by describing the boundary of $\Delta_{n,2}$. The hypersimplex $\Delta_{n,2}$ is an $(n-1)$-dimensional convex polytope, namely the convex hull of the vertices $\bm e_i + \bm e_j$ for all $1 \le i < j \le n$. In other words,
$$\Delta_{n,2} = I^n \cap \left\{ (x_1,\ldots, x_n) \in \mathbb{R}^n \;\middle|\; x_1 + \cdots + x_n = 2 \right\}.$$
Therefore,
$$\partial \Delta_{n,2} = \partial I^n \cap \left\{ (x_1,\ldots, x_n) \in \mathbb{R}^n \;\middle|\; x_1 + \cdots + x_n = 2 \right\}.$$
The boundary $\partial I^n$ consists of all points in $I^n$ where at least one coordinate is either $0$ or $1$. Intersecting with the hyperplane $x_1 + \cdots + x_n = 2$ yields $2n$ admissible $(n-2)$-dimensional polytopes:
\begin{itemize}
    \item For each $i$, setting $x_i = 0$ gives a hypersimplex $\Delta_{n-1,2}(i)$ of dimension $n-2$.
    \item For each $i$, setting $x_i = 1$ gives a simplex $\Delta_{n-1,1}(i)$ of dimension $n-2$.
\end{itemize}
Thus, $\partial \Delta_{n,2}$ consists of $n$ such hypersimplices and $n$ such simplices. Taking into account the hyperplane $x_1 + \cdots + x_n = 2$, the equations $x_i = 1$ can be rewritten as $\sum_{j \neq i} x_j = x_i$.

An important specific feature of admissible polytopes is the following.
\begin{theorem}\label{boundary_admissible}
    Let $P_\sigma$ be an admissible polytope of dimension $\dim P_\sigma \le n-3$. Then $P_\sigma$ is contained in the boundary $\partial \Delta_{n,2}$ for $\Delta_{n,2}$.
\end{theorem}
\begin{proof}
    By Theorem~\ref{dimensionadmissiblepolytope}, we have $\dim P_\sigma \le n-3$ whenever $|B| \ge 1$, where $B$ is the set of isolated vertices in the admissible graph $G_\sigma$. Let $P_\sigma$ be the moment polytope of the toric variety $\overline{\mathcal{O}_{\mathbb{C}}(L)}$ for some $L \in G_{n,2}$. Then $|B| \ge 1$ implies that there exists an index $m$ with $1 \le m \le n$ such that $P^{ij}(L) = 0$ for all $(i,j) \in \binom{[n]}{2}$ containing $m$. Consequently, $P_\sigma \subseteq \{ x_m = 0 \}$.
\end{proof}

Let $\mathcal{A}_n$ denote the hyperplane arrangement in $\mathbb{R}^n$ consisting of hyperplanes that pass through admissible polytopes $P_\sigma$ of codimension one in $\Delta_{n,2}$, i.e., of dimension $n-2$. Recall that we have a projection map (\ref{projectionmap}) which maps the simplex $\Delta^N \subset \mathbb{R}^{\binom{n}{2}}$ onto the hypersimplex $\Delta_{n,2} \subset \mathbb{R}^n$, and it sends the vertices of $\Delta^N$ bijectively onto the vertices of $\Delta_{n,2}$. We will find $\mathcal{A}_n$ using the dual monomorphism $\Gamma^* \colon (\mathbb{R}^n)^* \to \big( \mathbb{R}^{\binom{n}{2}} \big)^*$:

\begin{itemize}
    \item For a linear functional $\xi \in (\mathbb{R}^n)^*$ with coordinates $(\xi_1,\ldots, \xi_n)$, we have
\begin{equation}\label{gamma*}
    \Gamma^*(\xi_1,\ldots, \xi_n) = (\xi_1 + \xi_2,\; \ldots,\; \xi_i + \xi_j,\; \ldots,\; \xi_{n-1} + \xi_n)_{1 \le i < j \le n}.
\end{equation}

    \item Let $H \in \mathcal{A}_n$ be a hyperplane supporting $P_\sigma$, where $\dim P_\sigma = n-2$, and let $H$ be given by 
    $$H_{\xi} = \{ \bm x \in \mathbb{R}^n \mid \xi(\bm x) = 0 \}.$$
    We can lift it via $\Gamma$, i.e., consider $\Gamma^* \xi \in \big( \mathbb{R}^{\binom{n}{2}} \big)^*$. Since $\Gamma$ maps the vertices of $\Delta^N$ bijectively onto the vertices of $\Delta_{n,2}$, we conclude that 
    $$H_{\Gamma^* \xi} = \{ \bm y \in \mathbb{R}^{\binom{n}{2}} \mid \Gamma^* \xi(\bm y) = \xi(\Gamma \bm y) = 0 \}$$
    is a supporting hyperplane for a face $Q_{\sigma}$ of $\Delta^{N}$ given by
    $$Q_{\sigma} = \Delta^N \cap \{ (y_{12},\ldots, y_{n-1,n}) \in \mathbb{R}^{\binom{n}{2}} \mid y_{ij} = 0 \text{ for } (i,j) \in \sigma \},$$
    which is the intersection of $\Delta^N$ with the coordinate subspace where the $(i,j)$-th coordinate is zero for $(i,j) \in \sigma$. This face is not necessarily a facet.
    
    \item The hyperplane $H_{\Gamma^* \xi}$ is given by the linear functional $\Gamma^* \xi$. Therefore, by (\ref{gamma*}), we have $\xi_i + \xi_j = 0$ for all $(i,j) \in \sigma$.

    \item The dimension of $P_\sigma$ equals $n-2$. Hence, the linear functional $\xi$ is determined uniquely up to scaling. Thus, we need to find all admissible sets $\sigma$ for which the system of linear equations $\xi_i + \xi_j = 0$ for $(i,j) \in \sigma$ has a \emph{one-dimensional} solution space.

    \item By the classification theorem (Theorem~\ref{classification}), the admissible graph $G_{\sigma}$ is the disjoint union of a complete $N$-partite graph $K(A_1,\ldots, A_N)$ together with a set $B$ of isolated vertices. The edges of $G_\sigma$ correspond to the equations $\xi_i + \xi_j = 0$ for $(i,j) \in \sigma$. 

    \item Observe that the solution space is one-dimensional only in the following cases: $|B| = 1$ and $N = n-1$, or $|B| = 0$ and $N = 2$. The first case corresponds to the hyperplanes $x_i = 0$, which cut out the boundary hypersimplices $\Delta_{n-1,2}(i)$. The second case includes the hyperplanes $\sum_{j \neq i} x_j = x_i$, which cut out the boundary simplices $\Delta_{n-1,1}(i)$, as well as hyperplanes that intersect the interior of $\Delta_{n,2}$, given by equations of the form
    $$x_{i_1} + \cdots + x_{i_p} - x_{j_1} - \cdots - x_{j_q} = 0, \quad \text{where } p+q = n,\;\; p,q>1$$
\end{itemize}
We have shown that $\mathcal{A}_n$ consists of the following vector hyperplanes
\begin{equation}
    x_i = 0, \quad x_i - \sum_{j\ne i}x_j=0, \quad \sum^{p}_{k=1} x_{i_k} - \sum^{q}_{m=1} x_{j_m} = 0, \quad \text{where } p+q = n,\;\; p,q > 1.
\end{equation}

\begin{theorem}
    Admissible polytopes of dimension $n-2$ that are not contained in the boundary $\partial \Delta_{n,2}$ coincide with the intersections of the hypersimplex $\Delta_{n,2}$ with hyperplanes from $\mathcal{A}_n$ of the form
    \begin{equation}\label{H_p,q}
        H_{p,q}\;\colon \quad\sum^{p}_{k=1} x_{i_k} - \sum^{q}_{m=1} x_{j_m} = 0, \quad \text{where } p+q = n,\;\; p,q > 1.
    \end{equation}
\end{theorem}
\begin{proof}
   These hyperplanes cut out $(n-2)$-dimensional admissible polytopes whose admissible graph is a complete bipartite graph $K_{p,q}$ with $p+q = n$, $p,q > 1$. No proper subgraph of such a graph corresponds to an admissible polytope of the same dimension. Consequently, by definition of $\mathcal{A}_n$, these hyperplanes yield all $(n-2)$-dimensional admissible polytopes that are not contained in the boundary $\partial \Delta_{n,2}$.
\end{proof}

\begin{remark}
    The intersection of a hyperplane $H_{p,q}$ from the arrangement $\mathcal{A}_n$ with the affine plane $R^{n-1} = \{ x_1 + \ldots + x_n = 2\}$ yields an affine hyperplane $x_{i_1} + \ldots + x_{i_p} = 1$. Thus, we obtain the hyperplane arrangement of Buchstaber and Terzi\'c from \cite{BT3}.
\end{remark}

The admissible polytopes from the last theorem admit an explicit combinatorial description.
\begin{theorem}\label{(n-2)-dim_adm}
    Let $P_\sigma$ be an admissible polytope of dimension $n-2$ that intersects the interior of $\Delta_{n,2}$ and is cut out by a hyperplane $H_{p,q}$ with $p+q = n$, $p,q > 1$. Then $P_\sigma$ is combinatorially equivalent to $\Delta^{p-1} \times \Delta^{q-1}$, and the closure of the $(\mathbb{C}^*)^n$-orbit of a point $L \in W_\sigma$ is diffeomorphic to $\mathbb{C}P^{p-1} \times \mathbb{C}P^{q-1}$.
\end{theorem}
\begin{proof}
 Let $X$ be an $n \times 2$ matrix representing a $2$-plane $L \in W_\sigma$. Recall that $X$ defines $L$ up to the right $GL(2,\mathbb{C})$-action. The matrix $X$ gives a configuration of $n$ rows, i.e., vectors in $\mathbb{C}^2$. The admissible graph $G_\sigma$ is a complete bipartite graph $K_{p,q}$. In terms of the configuration of $n$ vectors in $\mathbb{C}^2$, this means that all vectors are nonzero and split into two groups, each consisting of pairwise proportional vectors. Using the $GL(2,\mathbb{C})$-action, we can normalize the configuration so that all vectors in the first group are proportional to $(1,0)$ and all vectors in the second group are proportional to $(0,1)$. Then the nonzero vectors take the form
$(z_1,0),\ \ldots,\ (z_p,0) \in \mathbb{C}^2$ and $(0,w_1),\ \ldots,\ (0,w_q) \in \mathbb{C}^2$, with $z_i, w_j \neq 0$.
After reordering the coordinates on $\mathbb{C}^n$ by the Weyl group, we obtain the following $n \times 2$ matrix:
$$
X = \begin{pmatrix}
z_1 & \dots & z_p & 0 & \dots & 0 \\
0   & \dots & 0   & w_1 & \dots & w_q
\end{pmatrix}^{\mathsf{T}}.
$$

Consider the action of the torus $(\mathbb{C}^*)^n$ on $L = [X]$. By taking limits, we may let the scalars $z_i$ and $w_j$ tend to $0$, provided that in each group at least one scalar remains nonzero. This leads to an $(\mathbb{C}^*)^n$-equivariant isomorphism
$$
\overline{\mathcal{O}_{\mathbb{C}}(L)} \;\cong\; \mathbb{C}P^{p-1} \times \mathbb{C}P^{q-1},
$$
where the induced $(\mathbb{C}^*)^n$-action on $\mathbb{C}P^{p-1} \times \mathbb{C}P^{q-1}$ is given by
$$
(t_{1},\ldots, t_{n}) \cdot ([z_{1}:\cdots : z_{p}], [w_{1}:\cdots : w_{q}]) = 
([t_{1}z_{1}:\cdots : t_{p}z_{p}],\, [t_{p+1}w_{1}:\cdots : t_{p+q}w_{q}]).
$$
If we quotient $(\mathbb{C}^*)^n$ by the stabilizer $\operatorname{Stab}(L)$, we obtain an effective action of $(\mathbb{C}^*)^{p+q-2}=(\mathbb{C}^*)^{p-1}\times (\mathbb{C}^{*})^{q-1}$ for which $\mathbb{C}P^{p-1}\times \mathbb{C}P^{q-1}$ is a toric variety with moment polytope $\Delta^{p-1}\times \Delta^{q-1}$.
\end{proof}

\begin{remark}
    A more general smooth result is known due to M. Noji and K. Ogiwara~\cite{NojiOgiwara}. They proved that for general $1\le k\le n$, every smooth closure of a $(\mathbb{C}^*)^n$-orbit in $G_{n,k}$ is a product (possibly of more than two) of projective spaces, and the corresponding simple polytope is a product of simplices.
\end{remark}

We now describe admissible polytopes of maximal dimension, i.e., of dimension $n-1$. We first need the following statement.

\begin{proposition}\label{boundaryofadmissible}
    A face of an admissible polytope in $\Delta_{n,2}$ is itself an admissible polytope.
\end{proposition}

\begin{proof}
    Indeed, the orbit closure $\overline{\mathcal{O}_{\mathbb{C}}(L)}$ contains orbits of smaller dimension on its boundary. Under the moment map $\mu$, these orbits project to the boundary faces of the moment polytope.
\end{proof}

Therefore, in order to describe an admissible polytope of dimension $n-1$, it is enough to describe its facets that intersect the interior of $\Delta_{n,2}$ and facets that lie in the boundary $\partial \Delta_{n,2}$.

\begin{theorem}\label{(n-1)-dim_adm}
Let $P_\sigma$ be an admissible polytope of dimension $n-1$ with admissible graph $G_\sigma = K(A_1,\ldots, A_N)$, $N \ge 3$. Then
\begin{enumerate}
    \item[(a)] The facets of $P_\sigma$ intersecting the interior of $\Delta_{n,2}$ are exactly the admissible polytopes $P_\tau$ with $G_\tau = K(A_k, [n] \setminus A_k)$ for those $A_k$ with $|A_k| > 1$.
    \item[(b)] The facets of $P_\sigma$ lying in the boundary $\partial \Delta_{n,2}$ are the admissible polytopes $P_\tau$ with $G_\tau = K(A_k, [n]\setminus A_k)$ for those $A_k$ with $|A_k| = 1$, as well as subgraphs of $G_\sigma$ that can be obtained by isolating a vertex so that the number of parts in the graph remains at least $3$.
\end{enumerate}
\end{theorem}
\begin{proof}
Let $t = (t_1,\ldots, t_n) \in (\mathbb{C}^*)^n$. For any $L \in G_{n,2}$,
$$P^{ij}(t \cdot L) = t_i t_j P^{ij}(L).$$

First, we exhibit a family of limits that yield facets intersecting the 
interior of $\Delta_{n,2}$. Choose a component $A_k$ with $|A_k| > 1$. 
Let $\varepsilon \to \infty$ and set $t_i = \varepsilon$ for $i \in A_k$, 
$t_j = 1$ for $j \notin A_k$. Dividing all coordinates in $\mathbb{C}P^{N_2}$ 
by $\varepsilon$, edges between $A_r$ and $A_s$ ($r,s \neq k$) vanish, 
while edges between $A_k$ and $A_r$ ($r \neq k$) survive. 
The limit graph $G_\tau$ is the complete bipartite graph $K(A_k, [n]\setminus A_k)$,
and the corresponding admissible polytope $P_\tau$ is a facet of $P_\sigma$ 
intersecting the interior of $\Delta_{n,2}$ (see Fig.~\ref{fig:limit_graph} for an example).

\begin{figure}[ht]
    \centering
    \tikzset{every picture/.style={line width=0.75pt}} 

\begin{tikzpicture}[x=0.75pt,y=0.75pt,yscale=-1,xscale=1]

\draw    (170.46,129.54) ;
\draw [shift={(170.46,129.54)}, rotate = 0] [color={rgb, 255:red, 0; green, 0; blue, 0 }  ][fill={rgb, 255:red, 0; green, 0; blue, 0 }  ][line width=0.75]      (0, 0) circle [x radius= 3.35, y radius= 3.35]   ;
\draw [shift={(170.46,129.54)}, rotate = 0] [color={rgb, 255:red, 0; green, 0; blue, 0 }  ][fill={rgb, 255:red, 0; green, 0; blue, 0 }  ][line width=0.75]      (0, 0) circle [x radius= 3.35, y radius= 3.35]   ;
\draw    (180.67,170.35) ;
\draw [shift={(180.67,170.35)}, rotate = 0] [color={rgb, 255:red, 0; green, 0; blue, 0 }  ][fill={rgb, 255:red, 0; green, 0; blue, 0 }  ][line width=0.75]      (0, 0) circle [x radius= 3.35, y radius= 3.35]   ;
\draw [shift={(180.67,170.35)}, rotate = 0] [color={rgb, 255:red, 0; green, 0; blue, 0 }  ][fill={rgb, 255:red, 0; green, 0; blue, 0 }  ][line width=0.75]      (0, 0) circle [x radius= 3.35, y radius= 3.35]   ;
\draw    (184.08,211.59) ;
\draw [shift={(184.08,211.59)}, rotate = 0] [color={rgb, 255:red, 0; green, 0; blue, 0 }  ][fill={rgb, 255:red, 0; green, 0; blue, 0 }  ][line width=0.75]      (0, 0) circle [x radius= 3.35, y radius= 3.35]   ;
\draw [shift={(184.08,211.59)}, rotate = 0] [color={rgb, 255:red, 0; green, 0; blue, 0 }  ][fill={rgb, 255:red, 0; green, 0; blue, 0 }  ][line width=0.75]      (0, 0) circle [x radius= 3.35, y radius= 3.35]   ;
\draw    (431.88,74.4) ;
\draw [shift={(431.88,74.4)}, rotate = 0] [color={rgb, 255:red, 0; green, 0; blue, 0 }  ][fill={rgb, 255:red, 0; green, 0; blue, 0 }  ][line width=0.75]      (0, 0) circle [x radius= 3.35, y radius= 3.35]   ;
\draw [shift={(431.88,74.4)}, rotate = 0] [color={rgb, 255:red, 0; green, 0; blue, 0 }  ][fill={rgb, 255:red, 0; green, 0; blue, 0 }  ][line width=0.75]      (0, 0) circle [x radius= 3.35, y radius= 3.35]   ;
\draw    (450.94,102.84) ;
\draw [shift={(450.94,102.84)}, rotate = 0] [color={rgb, 255:red, 0; green, 0; blue, 0 }  ][fill={rgb, 255:red, 0; green, 0; blue, 0 }  ][line width=0.75]      (0, 0) circle [x radius= 3.35, y radius= 3.35]   ;
\draw [shift={(450.94,102.84)}, rotate = 0] [color={rgb, 255:red, 0; green, 0; blue, 0 }  ][fill={rgb, 255:red, 0; green, 0; blue, 0 }  ][line width=0.75]      (0, 0) circle [x radius= 3.35, y radius= 3.35]   ;
\draw    (397.84,54.7) ;
\draw [shift={(397.84,54.7)}, rotate = 0] [color={rgb, 255:red, 0; green, 0; blue, 0 }  ][fill={rgb, 255:red, 0; green, 0; blue, 0 }  ][line width=0.75]      (0, 0) circle [x radius= 3.35, y radius= 3.35]   ;
\draw [shift={(397.84,54.7)}, rotate = 0] [color={rgb, 255:red, 0; green, 0; blue, 0 }  ][fill={rgb, 255:red, 0; green, 0; blue, 0 }  ][line width=0.75]      (0, 0) circle [x radius= 3.35, y radius= 3.35]   ;
\draw    (475.45,199.12) ;
\draw [shift={(475.45,199.12)}, rotate = 0] [color={rgb, 255:red, 0; green, 0; blue, 0 }  ][fill={rgb, 255:red, 0; green, 0; blue, 0 }  ][line width=0.75]      (0, 0) circle [x radius= 3.35, y radius= 3.35]   ;
\draw [shift={(475.45,199.12)}, rotate = 0] [color={rgb, 255:red, 0; green, 0; blue, 0 }  ][fill={rgb, 255:red, 0; green, 0; blue, 0 }  ][line width=0.75]      (0, 0) circle [x radius= 3.35, y radius= 3.35]   ;
\draw    (407.37,240.7) ;
\draw [shift={(407.37,240.7)}, rotate = 0] [color={rgb, 255:red, 0; green, 0; blue, 0 }  ][fill={rgb, 255:red, 0; green, 0; blue, 0 }  ][line width=0.75]      (0, 0) circle [x radius= 3.35, y radius= 3.35]   ;
\draw [shift={(407.37,240.7)}, rotate = 0] [color={rgb, 255:red, 0; green, 0; blue, 0 }  ][fill={rgb, 255:red, 0; green, 0; blue, 0 }  ][line width=0.75]      (0, 0) circle [x radius= 3.35, y radius= 3.35]   ;
\draw   (116,170.35) .. controls (116,136.93) and (144.96,109.84) .. (180.67,109.84) .. controls (216.39,109.84) and (245.35,136.93) .. (245.35,170.35) .. controls (245.35,203.76) and (216.39,230.85) .. (180.67,230.85) .. controls (144.96,230.85) and (116,203.76) .. (116,170.35) -- cycle ;
\draw  [dash pattern={on 0.84pt off 2.51pt}] (354.27,86.1) .. controls (354.27,52.69) and (383.23,25.6) .. (418.95,25.6) .. controls (454.66,25.6) and (483.62,52.69) .. (483.62,86.1) .. controls (483.62,119.52) and (454.66,146.61) .. (418.95,146.61) .. controls (383.23,146.61) and (354.27,119.52) .. (354.27,86.1) -- cycle ;
\draw  [dash pattern={on 0.84pt off 2.51pt}] (372.65,222.1) .. controls (372.65,188.68) and (401.61,161.59) .. (437.33,161.59) .. controls (473.04,161.59) and (502,188.68) .. (502,222.1) .. controls (502,255.51) and (473.04,282.6) .. (437.33,282.6) .. controls (401.61,282.6) and (372.65,255.51) .. (372.65,222.1) -- cycle ;
\draw    (170.46,129.54) -- (397.84,54.7) ;
\draw    (170.46,129.54) -- (431.88,74.4) ;
\draw    (170.46,129.54) -- (450.94,102.84) ;
\draw    (180.67,170.35) -- (397.84,54.7) ;
\draw    (180.67,170.35) -- (431.88,74.4) ;
\draw    (180.67,170.35) -- (450.94,102.84) ;
\draw    (184.08,211.59) -- (397.84,54.7) ;
\draw    (184.08,211.59) -- (431.88,74.4) ;
\draw    (184.08,211.59) -- (450.94,102.84) ;
\draw    (184.08,211.59) -- (407.37,240.7) ;
\draw    (184.08,211.59) -- (475.45,199.12) ;
\draw    (180.67,170.35) -- (407.37,240.7) ;
\draw    (180.67,170.35) -- (475.45,199.12) ;
\draw    (170.46,129.54) -- (407.37,240.7) ;
\draw    (170.46,129.54) -- (475.45,199.12) ;
\draw  [dash pattern={on 0.84pt off 2.51pt}]  (475.45,199.12) -- (450.94,102.84) ;
\draw  [dash pattern={on 0.84pt off 2.51pt}]  (475.45,199.12) -- (431.88,74.4) ;
\draw  [dash pattern={on 0.84pt off 2.51pt}]  (475.45,199.12) -- (397.84,54.7) ;
\draw  [dash pattern={on 0.84pt off 2.51pt}]  (407.37,240.7) -- (450.94,102.84) ;
\draw  [dash pattern={on 0.84pt off 2.51pt}]  (407.37,240.7) -- (431.88,74.4) ;
\draw  [dash pattern={on 0.84pt off 2.51pt}]  (407.37,240.7) -- (397.84,54.7) ;

\draw (126,161.4) node [anchor=north west][inner sep=0.75pt]    {$A_{k}$};
\draw (431,32.4) node [anchor=north west][inner sep=0.75pt]    {$A_{r}$};
\draw (455,244.4) node [anchor=north west][inner sep=0.75pt]    {$A_{s}$};
\draw (153,123.4) node [anchor=north west][inner sep=0.75pt]    {$1$};
\draw (164,159.4) node [anchor=north west][inner sep=0.75pt]    {$2$};
\draw (168,205.4) node [anchor=north west][inner sep=0.75pt]    {$3$};
\draw (404,38.4) node [anchor=north west][inner sep=0.75pt]    {$4$};
\draw (439,63.4) node [anchor=north west][inner sep=0.75pt]    {$5$};
\draw (461,90.4) node [anchor=north west][inner sep=0.75pt]    {$6$};
\draw (483,198.4) node [anchor=north west][inner sep=0.75pt]    {$7$};
\draw (409.37,244.1) node [anchor=north west][inner sep=0.75pt]    {$8$};

\end{tikzpicture}
    \caption{The limit graph $G_\tau = K(A_k, [n]\setminus A_k)$ for $A_k = \{1,2,3\}$. 
Edges between $\{4,5,6\}$ and $\{7,8\}$ vanish as $\varepsilon \to \infty$; edges incident to $A_k$ survive.}
    \label{fig:limit_graph}
\end{figure}
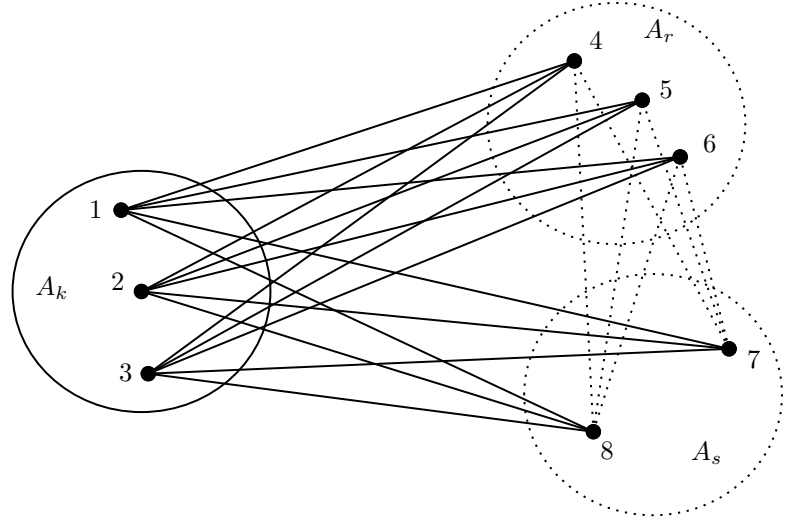

We now show that all other limits yield graphs with at least one isolated 
vertex or a star graph, and therefore the corresponding polytopes lie on 
the boundary $\partial \Delta_{n,2}$.

\emph{Case 1.} None of the $t_i$ tend to $\infty$. 
If some $t_j \to 0$, then $j$ becomes an isolated vertex.

\emph{Case 2.} There exist $i \in A_p$, $j \in A_q$ ($p \neq q$) with 
$t_i, t_j \to \infty$. Choose a pair $(i,j) \in \sigma$ among all such 
vertices for which $t_i t_j \to \infty$ no slower than for any other pair. 
Then $P^{ij}(t \cdot L) \to \infty$. Dividing all coordinates by $t_i t_j$, 
all vertices in components $A_k$ with $k \neq p,q$ become isolated.

\emph{Case 3.} All $t_i \to \infty$ lie in one component $A_k$. 
Choose $i_0 \in A_k$ with maximal order. Dividing all coordinates by 
$t_{i_0}$, we may assume $t_j = 1$ for $j \notin A_k$ (otherwise $j$ 
becomes isolated). Vertices in $A_k$ of smaller order become isolated. 
For the remaining vertices, edges between $A_r$ and $A_s$ ($r,s \neq k$) 
vanish. The limit graph is a complete bipartite graph, possibly with 
isolated vertices. If $|A_k| = 1$, we obtain a star graph.

Thus, the only limits without isolated vertices and with $N \ge 2$ are 
the complete bipartite graphs $K(A_k, [n]\setminus A_k)$ where $|A_k| > 1$. 
These correspond precisely to the facets of $P_\sigma$ that intersect the 
interior of $\Delta_{n,2}$.

It remains to observe that the subgraphs $G_\tau$ of $G_\sigma$ that determine facets of $P_\sigma$ are either star graphs, in which case $P_\tau$ is a boundary simplex, or they cut out an admissible $(n-2)$-dimensional polytope in a copy of the hypersimplex $\Delta_{n-1,2}$ inside the boundary $\partial \Delta_{n,2}$. The latter corresponds to isolating exactly one vertex, provided that the number of parts is at least $3$, in accordance with Theorem~\ref{dimensionadmissiblepolytope}.
\end{proof}

\begin{corollary}
Let $P_\sigma$ be an admissible polytope of dimension $n-1$ with admissible graph $G_\sigma = K(A_1,\ldots, A_N)$, $N\ge 3$. Then $P_\sigma$ lies in the intersection of $\Delta_{n,2}$ with the closed half-spaces determined by the hyperplanes (\ref{H_p,q}) that correspond to the graphs $K(A_k, [n] \setminus A_k)$ for $|A_k| > 1$.
\end{corollary}

The arrangement of hyperplanes $\mathcal{A}_n$ in $\mathbb{R}^n$ defines an \emph{intersection lattice} $L(\mathcal{A}_n)$. By definition (see \cite{Stanley}), the intersection lattice $L(\mathcal{A}_n)$ is the set of all nonempty intersections of hyperplanes in $\mathcal{A}_n$, including $\mathbb{R}^n$ itself as the intersection over the empty set. A \emph{region} of the arrangement $\mathcal{A}_n$ is a connected component of the complement $X$ of the hyperplanes:
$$X = \mathbb{R}^n \setminus \bigcup_{H \in \mathcal{A}_n} H.$$
The set of all regions is denoted by $R(\mathcal{A}_n)$. A (closed) \emph{face} $F \subseteq \mathbb{R}^n$ of the arrangement $\mathcal{A}_n$ is a nonempty set of the form
$$F = \overline{R} \cap x,\quad R \in R(\mathcal{A}_n),\; x \in L(\mathcal{A}_n).$$
In particular, if $x = \varnothing$, then $\overline{R}$ itself is also considered as a face. The set of all faces is denoted by $F(\mathcal{A}_n)$. Note that 
$$\mathbb{R}^n = \bigsqcup_{F \in F(\mathcal{A}_n)} \mathring{F}, \quad \text{where } \mathring{F} = \operatorname{relint} F.$$

\begin{proposition}[\cite{BT3}]\label{arrangement}
    For each chamber $C$ lying in the interior of $\Delta_{n,2}$, there exists a face $F \in F(\mathcal{A}_n)$ of the hyperplane arrangement $\mathcal{A}_n$ such that
$$C = \mathring{F} \cap \Delta_{n,2}.$$
Conversely, any face $F \in F(\mathcal{A}_n)$ whose intersection with $\Delta_{n,2}$ is nonempty determines a chamber $C = \mathring{F} \cap \Delta_{n,2}$.
\end{proposition}
\begin{proof}
If $\dim C = n-1$, then $C$ is a connected component of the complement in $\Delta_{n,2}$ of the union of all hyperplanes from $\mathcal{A}_n$. Indeed, by definition, $C$ is the intersection of the interiors of all $(n-1)$-dimensional admissible polytopes that contain it. By the previous theorem, the boundary of each such polytope is formed by hyperplanes from $\mathcal{A}_n$. Therefore, $C$ avoids all these hyperplanes and is maximal with respect to this property. Since the complement $\mathbb{R}^n \setminus \bigcup_{H \in \mathcal{A}_n} H$ is a union of open regions $R \in R(\mathcal{A}_n)$, and $C$ is an open connected subset of $\Delta_{n,2}$ that does not meet any hyperplane, $C$ must coincide with the intersection of exactly one such region $R$ with the hypersimplex:
$$C = R \cap \Delta_{n,2} = \mathring{R} \cap \Delta_{n,2}.$$
Thus, $C = \mathring{F} \cap \Delta_{n,2}$ for the face $F = \overline{R}$ of the arrangement $\mathcal{A}_n$, which completes the proof for the maximal-dimensional case. 

If $\dim C < n-1$, then by definition of the chamber decomposition, $C$ is the intersection 
of the relative interiors of all admissible polytopes that contain it. By assumption, $C$ lies 
in the interior of the hypersimplex $\Delta_{n,2}$. By Theorem~\ref{boundary_admissible}, 
every admissible polytope containing $C$ has dimension at least $n-2$. Since $\dim C < n-1$, 
this intersection must involve at least one admissible polytope of dimension $n-2$; the 
relative interior of each such polytope lies in a hyperplane $H \in \mathcal{A}_n$. 
Let $\mathcal{H}_C = \{ H \in \mathcal{A}_n \mid C \subseteq H \}$ and let 
$L_C = \bigcap_{H \in \mathcal{H}_C} H \in L(\mathcal{A}_n)$ be their intersection. 
The $(n-1)$-dimensional admissible polytopes containing $C$ (if any) serve only to cut out 
$C$ within $L_C$ via their boundary hyperplanes. Consequently, $C$ avoids all hyperplanes 
of $\mathcal{A}_n$ not in $\mathcal{H}_C$, and is an open connected component of 
$L_C \setminus \bigcup_{H \notin \mathcal{H}_C} H$ relative to $L_C$. Such a component 
is precisely the relative interior $\mathring{F}$ of some face $F \in F(\mathcal{A}_n)$. 
Hence $C = \mathring{F} \cap \Delta_{n,2}$.

The converse follows by reversing the argument: a face $F \in F(\mathcal{A}_n)$ with 
$\mathring{F} \cap \Delta_{n,2} \neq \varnothing$ gives a connected component of the 
complement of all hyperplanes not containing it, which is exactly a chamber.
\end{proof}

\begin{example}
    The hypersimplex $\Delta_{4,2}$ is an octahedron lying in the hyperplane $x_1+x_2 + x_3 + x_4 = 2$ in $\mathbb{R}^4$. There are $8$ chambers determined by the hyperplanes associated with $8$ boundary triangles and $3$ diagonal squares; see Fig.~\ref{fig:chambers}. 

\begin{figure}[ht]
    \centering
    \input{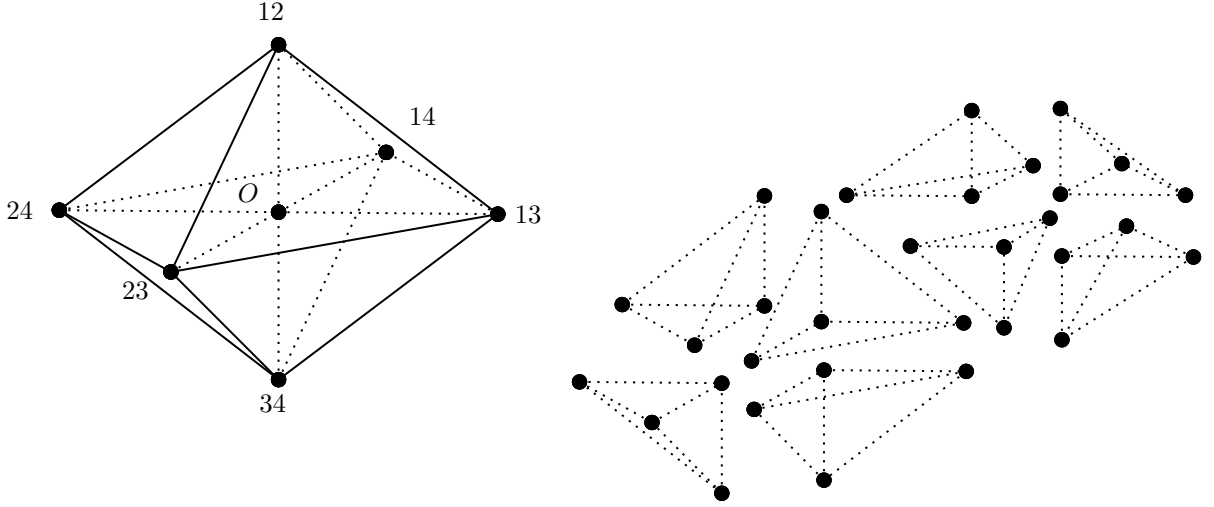}
    \caption{Three-dimensional chambers of $\Delta_{4,2}$.}
    \label{fig:chambers}
\end{figure}

The admissible sets corresponding to the diagonal squares are
$$
\big\{(1,3), (1,4), (2,3), (2,4)\big\},\;
\big\{(1,2), (1,4), (2,3), (3,4)\big\},\;
\big\{(1,2), (1,3), (2,4), (3,4)\big\}.
$$
The corresponding admissible graphs $G_{\sigma}$ are the complete bipartite graph $K_{2,2}$; they are sketched in the Fig. \ref{fig:K_2,2}.

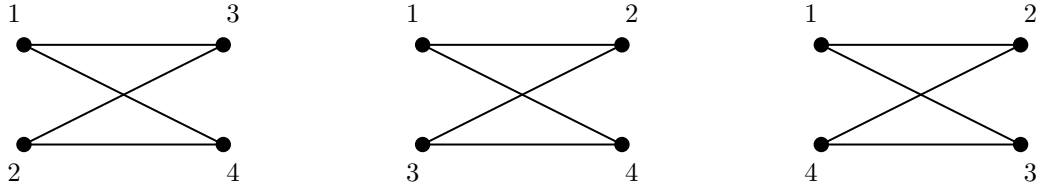
\begin{figure}[ht]
    \centering
    \tikzset{every picture/.style={line width=0.75pt}} 

\begin{tikzpicture}[x=0.75pt,y=0.75pt,yscale=-1,xscale=1]

\draw [fill=black] (100,100) circle [radius=3.35];
\draw [fill=black] (100,150) circle [radius=3.35];
\draw (95,75) node [anchor=north] {$1$};   
\draw (95,155) node [anchor=north] {$2$};  

\draw [fill=black] (200,100) circle [radius=3.35];
\draw [fill=black] (200,150) circle [radius=3.35];
\draw (205,75) node [anchor=north] {$3$};  
\draw (205,155) node [anchor=north] {$4$}; 

\draw (100,100) -- (200,100);
\draw (100,100) -- (200,150);
\draw (100,150) -- (200,100);
\draw (100,150) -- (200,150);

\draw [fill=black] (300,100) circle [radius=3.35];
\draw [fill=black] (300,150) circle [radius=3.35];
\draw (295,75) node [anchor=north] {$1$};  
\draw (295,155) node [anchor=north] {$3$}; 

\draw [fill=black] (400,100) circle [radius=3.35];
\draw [fill=black] (400,150) circle [radius=3.35];
\draw (405,75) node [anchor=north] {$2$};  
\draw (405,155) node [anchor=north] {$4$}; 

\draw (300,100) -- (400,100);
\draw (300,100) -- (400,150);
\draw (300,150) -- (400,100);
\draw (300,150) -- (400,150);

\draw [fill=black] (500,100) circle [radius=3.35];
\draw [fill=black] (500,150) circle [radius=3.35];
\draw (495,75) node [anchor=north] {$1$};  
\draw (495,155) node [anchor=north] {$4$}; 

\draw [fill=black] (600,100) circle [radius=3.35];
\draw [fill=black] (600,150) circle [radius=3.35];
\draw (605,75) node [anchor=north] {$2$};  
\draw (605,155) node [anchor=north] {$3$}; 

\draw (500,100) -- (600,100);
\draw (500,100) -- (600,150);
\draw (500,150) -- (600,100);
\draw (500,150) -- (600,150);

\end{tikzpicture}
    \caption{Admissible graphs for the diagonal squares.}
    \label{fig:K_2,2}
\end{figure}

According to Theorem~\ref{(n-2)-dim_adm}, the associated hyperplanes $H_{p,q}$ in $\mathbb{R}^4$ are
$$
x_1 + x_2 - x_3 - x_4 = 0,\quad x_1 - x_2 + x_3 - x_4 = 0,\quad x_1  - x_2 - x_3 + x_4 = 0.
$$
Intersecting each with the hyperplane $x_1 + x_2 + x_3 + x_4 = 2$ yields the affine hyperplanes
$$
x_1 + x_2 = 1,\quad x_1 + x_3 = 1,\quad x_1 + x_4 = 1.
$$

According to Theorem~\ref{(n-1)-dim_adm}, there are $6$ three-dimensional admissible polytopes given by a square pyramids. The corresponding admissible graphs $G_\sigma$ are sketched in the Fig. \ref{fig:pyramids}. Each of these graphs is a complete $3$-partite graph.

\begin{figure}[ht]
    \centering
    \input{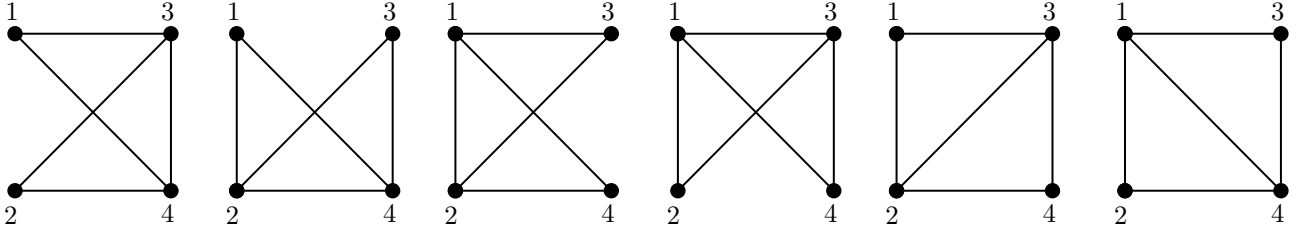}
    \caption{Admissible graphs for the square pyramids.}
    \label{fig:pyramids}
\end{figure}
\end{example}

We now illustrate our results for the case $n = 5$.

\begin{example}
The hypersimplex $\Delta_{5,2}$ is a $4$-dimensional convex polytope lying in $\mathbb{R}^5$. Up to combinatorial isomorphism, there are $4$ types of admissible polytopes of dimension $4$, given by all partitions of $5$ with at least three parts: $1+1+1+1+1$ (the complete graph corresponding to $\Delta_{5,2}$), together with
$$2+1+1+1,\quad 2+2+1,\quad 3+1+1.$$
For these partitions, the numbers of polytopes are $1$, $10$, $15$, and $10$, respectively. This yields a total of $36$ admissible polytopes of full dimension.

Admissible polytopes of dimension $3$ that intersect the interior of $\Delta_{5,2}$ correspond to the partition $3+2$. There are exactly $10$ such admissible polytopes.
\end{example}


\section{Chamber decomposition of $\Delta_{n,2}$ for the $T^n$-action on $\mathbb{C}P^{N_2}$}

In this section, we describe the chamber decomposition of $\Delta_{n,2}$ with respect to the $T^n$-action on $\mathbb{C}P^{N_2}$ induced by the second exterior power of the standard representation of $T^n$ on $\mathbb{C}^n$. We show that the cones in $(\mathfrak t^n)^* \cong \mathbb{R}^n$ spanned by the chambers form the GKZ decomposition of $\cone \Gamma$, where $\Gamma = \{ \bm e^i + \bm e^j \mid 1 \le i < j \le n\}$.

We introduce a hyperplane arrangement $\widetilde{\mathcal{A}}_n$ in $(\mathfrak t^n)^* \cong \mathbb{R}^n$ that determines the walls of the chambers of maximal dimension in $\Delta_{n,2}$ (i.e., of dimension $n-1$). Using this arrangement, we also give a combinatorial description of the polytopes $P_\sigma$ for $\sigma \subseteq \binom{[n]}{2}$ and describe the regular values of the moment map $\widetilde \mu$.

\subsection{GKZ Fan and Chamber Decomposition of $\Delta_{n,2}$}

Let $V$ be a real $k$-dimensional vector space and let $\Gamma = \{ \gamma_1,\ldots, \gamma_m\}$ be a spanning vector configuration in $V^*$. For $\sigma \subseteq [m]$, denote $\Gamma_{\sigma} = \{ \gamma_j \mid j \in \sigma\}$. The cones $\operatorname{cone} \Gamma_{\sigma}$ for $\sigma \subseteq [m]$ do not generally form a fan, since they may intersect in their relative interiors. The \emph{GKZ fan} $\Sigma_{\Gamma}$ is obtained by taking all possible maximal nonempty intersections of these cones. More precisely, for any $\delta \in \operatorname{cone} \Gamma$, set
\begin{equation}
    C(\delta) := \bigcap_{\substack{\sigma \subseteq [m] \\ \delta \in \cone \Gamma_{\sigma}}} \cone \Gamma_{\sigma} = \bigcap_{\substack{\sigma \subseteq [m] \\ \delta \in \operatorname{relint} \cone \Gamma_{\sigma}}} \cone \Gamma_{\sigma}.
\end{equation}
The cones $C(\delta)$ form a fan $\Sigma_{\Gamma}$ with support $\operatorname{cone} \Gamma$, and are called its \emph{chambers}; see \cite{Panov} or \cite{Arzhantsev, GelfandKapranovZelevinsky}.

Recall that the second exterior power $\Lambda^2 \mathbb{C}^n$ of the standard representation of the torus $T^n$ on $\mathbb{C}^n$ defines a configuration $\Gamma = \{ \bm e^i + \bm e^j \mid 1 \le i < j \le n\}$ in the space $V^* = (\mathfrak{t}^n)^*$, which we identify with $\mathbb{R}^n$ using the canonical basis given by the standard one-parameter subgroups of $T^n$.

The configuration $\Gamma$ coincides with the set of vertices of the hypersimplex $\Delta_{n,2}$, and $\cone \Gamma_{\sigma}$ is spanned by the vertices $\bm e^i + \bm e^j$ for all $(i,j) \in \sigma$, where $\sigma \subseteq \binom{[n]}{2}$ is an arbitrary set of $2$-indices (not necessarily admissible). Note that $P_\sigma = \cone \Gamma_{\sigma} \cap \Delta_{n,2}$. Let $\delta \in \Delta_{n,2}$ and let $\omega = \omega(\delta)$ be the collection of sets $\sigma \subseteq \binom{[n]}{2}$ for which $\delta \in \mathring{P}_\sigma$. It follows directly from the definition of the chamber decomposition of the hypersimplex $\Delta_{n,2}$ with respect to the $T^n$-action on $\mathbb{C}P^{N_2}$ that $C_{\omega} = C(\delta) \cap \Delta_{n,2}$. Thus, we have

\begin{remark}
    The chamber decomposition of $\Delta_{n,2}$ with respect to the $T^n$-action on $\mathbb{C}P^{N_2}$ coincides with the intersection of the hypersimplex $\Delta_{n,2}$ with the GKZ decomposition of $\cone \Gamma$ determined by the fan $\Sigma_{\Gamma}$.
\end{remark}

As is well known \cite{GelfandKapranovZelevinsky, Panov}, the GKZ fan $\Sigma_{\Gamma}$ encodes generalised normal fans with cone
generators chosen among the Gale dual configuration $A$.

\subsection{New Arrangement of Hyperplanes}

In Section~\ref{sectionGrassmannian}, we defined the arrangement $\mathcal{A}_n$ of hyperplanes in $(\mathfrak{t}^n)^* \cong \mathbb{R}^n$ that pass through admissible polytopes of codimension one in $\Delta_{n,2}$, i.e., of dimension $n-2$. Recall that the facets of $\Delta_{n,2}$ themselves are admissible. We now introduce a \emph{new arrangement} $\widetilde{\mathcal{A}}_n$ consisting of hyperplanes in $\mathbb{R}^n$ that pass through (not necessarily admissible) polytopes $P_\sigma$ of dimension $n-2$. In the following theorem, we describe the corresponding graph $G_\sigma$ defined in Section~\ref{sectionGrassmannian} as a graph with $V(G_\sigma) = [n]$ and $E(G_\sigma) = \sigma$ in the case where $P_\sigma$ intersects the interior of the hypersimplex $\Delta_{n,2}$ and is maximal with respect to inclusion among such polytopes.

\begin{theorem}
    For an arbitrary $\sigma$, the polytope $P_\sigma$ has dimension $n-2$ and is not contained in any larger polytope of dimension $n-2$ if and only if $G_\sigma$ consists of two connected components: a complete graph $K_l$ and a complete bipartite graph $K_{p,q}$, where either $l = 0$ and $p,q >1$, or $l \ge 3$ and $p,q \ge 1$.
\end{theorem}
\begin{proof}
Reasoning similarly to Section~\ref{sectionGrassmannian}, we see that such polytopes $P_\sigma$ correspond precisely to linear functionals $\xi \in (\mathbb{R}^n)^*$ with coordinates $(\xi_1,\ldots, \xi_n)$ such that the system of equations $\xi_i + \xi_j = 0$ for $(i,j) \in \sigma$ has a one-dimensional solution space.

Suppose there exists a vertex $i_0 \in \{1,\ldots, n\}$ such that $\xi_{i_0} = 0$. Consider all vertices $i \in \{1,\ldots, n\}$ for which $\xi_i = 0$. It follows directly from the definition of $G_\sigma$ that the induced subgraph $K$ of $G_\sigma$ on these vertices is a complete graph $K=K_l$ and a connected component of $G_\sigma$.

Now, consider a connected component $C$ of $G_\sigma$ different from this complete graph. Let $i_1 \in \{1,\ldots,n\}$ be a vertex of $C$. Then $\xi_{i_1} \neq 0$. Let $\{i_2,\ldots, i_p\} \subseteq \{1,\ldots,n\}$ be all vertices of $C$ that are not adjacent to $i_1$. From the connectivity of $C$ and the definition of $G_\sigma$, it follows that $\xi_{i_1} = \xi_{i_2} = \cdots = \xi_{i_p}$. Let $\{j_1,\ldots, j_q\} \subseteq \{1,\ldots,n\}$ be all vertices of $C$ adjacent to $i_1$. Then $\xi_{j_1} = \cdots = \xi_{j_q} = -\xi_{i_1} = -\xi_{i_2} = \cdots = -\xi_{i_p}$. Therefore, $C$ is a complete bipartite graph $K_{p,q}$. If $G_\sigma$ has more than one bipartite connected component, then the covector $\xi$ is not uniquely determined, since it would belong to a subspace of $(\mathbb{R}^n)^*$ of dimension at least $2$. An example of such a graph $G_\sigma$ is shown in Fig.~\ref{fig:mygraph}.

It remains only to observe that if $l = 1$, then the resulting polytope $P_\sigma$ lies on the boundary $\partial \Delta_{n,2}$. If $l = 2$, the complete component $K_l$ has only one edge and yields a new nonzero parameter. Hence we obtain an alternative: either $l = 0$ and $p,q >1$, or $l \ge 3$ and $p,q \ge 1$.
\end{proof}

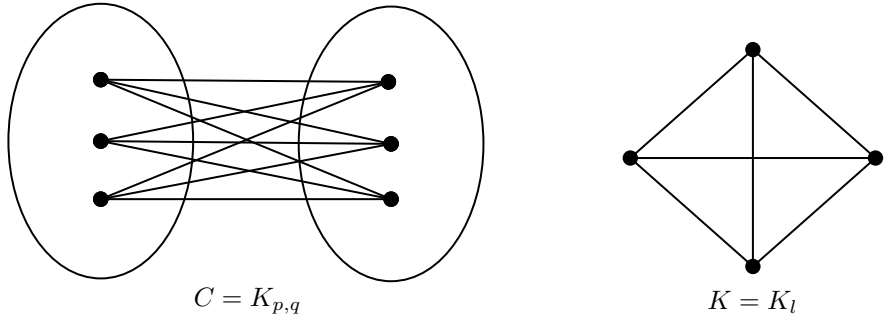
\begin{figure}[ht]
    \centering
    \tikzset{every picture/.style={line width=0.75pt}} 

\begin{tikzpicture}[x=0.75pt,y=0.75pt,yscale=-1,xscale=1]

\draw    (160.34,150.62) -- (305.98,122.9) ;
\draw [shift={(305.98,122.9)}, rotate = 349.22] [color={rgb, 255:red, 0; green, 0; blue, 0 }  ][fill={rgb, 255:red, 0; green, 0; blue, 0 }  ][line width=0.75]      (0, 0) circle [x radius= 3.35, y radius= 3.35]   ;
\draw [shift={(160.34,150.62)}, rotate = 349.22] [color={rgb, 255:red, 0; green, 0; blue, 0 }  ][fill={rgb, 255:red, 0; green, 0; blue, 0 }  ][line width=0.75]      (0, 0) circle [x radius= 3.35, y radius= 3.35]   ;
\draw    (304.65,91.81) -- (160.34,90.69) ;
\draw [shift={(160.34,90.69)}, rotate = 180.44] [color={rgb, 255:red, 0; green, 0; blue, 0 }  ][fill={rgb, 255:red, 0; green, 0; blue, 0 }  ][line width=0.75]      (0, 0) circle [x radius= 3.35, y radius= 3.35]   ;
\draw [shift={(304.65,91.81)}, rotate = 180.44] [color={rgb, 255:red, 0; green, 0; blue, 0 }  ][fill={rgb, 255:red, 0; green, 0; blue, 0 }  ][line width=0.75]      (0, 0) circle [x radius= 3.35, y radius= 3.35]   ;
\draw    (160.34,90.69) -- (305.98,122.9) ;
\draw [shift={(305.98,122.9)}, rotate = 12.47] [color={rgb, 255:red, 0; green, 0; blue, 0 }  ][fill={rgb, 255:red, 0; green, 0; blue, 0 }  ][line width=0.75]      (0, 0) circle [x radius= 3.35, y radius= 3.35]   ;
\draw [shift={(160.34,90.69)}, rotate = 12.47] [color={rgb, 255:red, 0; green, 0; blue, 0 }  ][fill={rgb, 255:red, 0; green, 0; blue, 0 }  ][line width=0.75]      (0, 0) circle [x radius= 3.35, y radius= 3.35]   ;
\draw   (114,121.5) .. controls (114,83.45) and (134.75,52.6) .. (160.34,52.6) .. controls (185.93,52.6) and (206.68,83.45) .. (206.68,121.5) .. controls (206.68,159.55) and (185.93,190.39) .. (160.34,190.39) .. controls (134.75,190.39) and (114,159.55) .. (114,121.5) -- cycle ;
\draw    (305.98,122.9) -- (160.34,121.5) ;
\draw [shift={(160.34,121.5)}, rotate = 180.55] [color={rgb, 255:red, 0; green, 0; blue, 0 }  ][fill={rgb, 255:red, 0; green, 0; blue, 0 }  ][line width=0.75]      (0, 0) circle [x radius= 3.35, y radius= 3.35]   ;
\draw [shift={(305.98,122.9)}, rotate = 180.55] [color={rgb, 255:red, 0; green, 0; blue, 0 }  ][fill={rgb, 255:red, 0; green, 0; blue, 0 }  ][line width=0.75]      (0, 0) circle [x radius= 3.35, y radius= 3.35]   ;
\draw   (259.64,122.9) .. controls (259.64,84.85) and (280.39,54) .. (305.98,54) .. controls (331.57,54) and (352.32,84.85) .. (352.32,122.9) .. controls (352.32,160.95) and (331.57,191.79) .. (305.98,191.79) .. controls (280.39,191.79) and (259.64,160.95) .. (259.64,122.9) -- cycle ;
\draw    (426,129.99) -- (549,129.99) ;
\draw [shift={(549,129.99)}, rotate = 0] [color={rgb, 255:red, 0; green, 0; blue, 0 }  ][fill={rgb, 255:red, 0; green, 0; blue, 0 }  ][line width=0.75]      (0, 0) circle [x radius= 3.35, y radius= 3.35]   ;
\draw [shift={(426,129.99)}, rotate = 0] [color={rgb, 255:red, 0; green, 0; blue, 0 }  ][fill={rgb, 255:red, 0; green, 0; blue, 0 }  ][line width=0.75]      (0, 0) circle [x radius= 3.35, y radius= 3.35]   ;
\draw    (487.5,75.6) -- (487.5,184.39) ;
\draw [shift={(487.5,184.39)}, rotate = 90] [color={rgb, 255:red, 0; green, 0; blue, 0 }  ][fill={rgb, 255:red, 0; green, 0; blue, 0 }  ][line width=0.75]      (0, 0) circle [x radius= 3.35, y radius= 3.35]   ;
\draw [shift={(487.5,75.6)}, rotate = 90] [color={rgb, 255:red, 0; green, 0; blue, 0 }  ][fill={rgb, 255:red, 0; green, 0; blue, 0 }  ][line width=0.75]      (0, 0) circle [x radius= 3.35, y radius= 3.35]   ;
\draw    (305.98,150.62) -- (160.34,121.5) ;
\draw [shift={(160.34,121.5)}, rotate = 191.31] [color={rgb, 255:red, 0; green, 0; blue, 0 }  ][fill={rgb, 255:red, 0; green, 0; blue, 0 }  ][line width=0.75]      (0, 0) circle [x radius= 3.35, y radius= 3.35]   ;
\draw [shift={(305.98,150.62)}, rotate = 191.31] [color={rgb, 255:red, 0; green, 0; blue, 0 }  ][fill={rgb, 255:red, 0; green, 0; blue, 0 }  ][line width=0.75]      (0, 0) circle [x radius= 3.35, y radius= 3.35]   ;
\draw    (160.34,150.62) -- (305.98,150.62) ;
\draw [shift={(305.98,150.62)}, rotate = 0] [color={rgb, 255:red, 0; green, 0; blue, 0 }  ][fill={rgb, 255:red, 0; green, 0; blue, 0 }  ][line width=0.75]      (0, 0) circle [x radius= 3.35, y radius= 3.35]   ;
\draw [shift={(160.34,150.62)}, rotate = 0] [color={rgb, 255:red, 0; green, 0; blue, 0 }  ][fill={rgb, 255:red, 0; green, 0; blue, 0 }  ][line width=0.75]      (0, 0) circle [x radius= 3.35, y radius= 3.35]   ;
\draw    (305.98,150.62) -- (160.34,90.69) ;
\draw [shift={(160.34,90.69)}, rotate = 202.37] [color={rgb, 255:red, 0; green, 0; blue, 0 }  ][fill={rgb, 255:red, 0; green, 0; blue, 0 }  ][line width=0.75]      (0, 0) circle [x radius= 3.35, y radius= 3.35]   ;
\draw [shift={(305.98,150.62)}, rotate = 202.37] [color={rgb, 255:red, 0; green, 0; blue, 0 }  ][fill={rgb, 255:red, 0; green, 0; blue, 0 }  ][line width=0.75]      (0, 0) circle [x radius= 3.35, y radius= 3.35]   ;
\draw    (304.65,91.81) -- (160.34,121.5) ;
\draw [shift={(160.34,121.5)}, rotate = 168.38] [color={rgb, 255:red, 0; green, 0; blue, 0 }  ][fill={rgb, 255:red, 0; green, 0; blue, 0 }  ][line width=0.75]      (0, 0) circle [x radius= 3.35, y radius= 3.35]   ;
\draw [shift={(304.65,91.81)}, rotate = 168.38] [color={rgb, 255:red, 0; green, 0; blue, 0 }  ][fill={rgb, 255:red, 0; green, 0; blue, 0 }  ][line width=0.75]      (0, 0) circle [x radius= 3.35, y radius= 3.35]   ;
\draw    (304.65,91.81) -- (160.34,150.62) ;
\draw [shift={(160.34,150.62)}, rotate = 157.83] [color={rgb, 255:red, 0; green, 0; blue, 0 }  ][fill={rgb, 255:red, 0; green, 0; blue, 0 }  ][line width=0.75]      (0, 0) circle [x radius= 3.35, y radius= 3.35]   ;
\draw [shift={(304.65,91.81)}, rotate = 157.83] [color={rgb, 255:red, 0; green, 0; blue, 0 }  ][fill={rgb, 255:red, 0; green, 0; blue, 0 }  ][line width=0.75]      (0, 0) circle [x radius= 3.35, y radius= 3.35]   ;
\draw   (487.5,75.6) -- (549,129.99) -- (487.5,184.39) -- (426,129.99) -- cycle ;

\draw (205.45,193.8) node [anchor=north west][inner sep=0.75pt]    {$C=K_{p,q}$};
\draw (463.17,195.2) node [anchor=north west][inner sep=0.75pt]    {$K=K_{l}$};

\end{tikzpicture}
    \caption{Connected components of the graph $G_{\sigma}$.}
    \label{fig:mygraph}
\end{figure}

We have shown that $\widetilde{\mathcal{A}}_n$ consists of the hyperplanes of $\mathcal{A}_n$ together with the following series of hyperplanes for $n \ge 5$:
\begin{equation}
   \sum^{p}_{k=1} x_{i_k} - \sum^{q}_{m=1} x_{j_m} = 0, \quad \text{where } p + q \le n-3,\;\; p,q \ge 1.
\end{equation}
Therefore, $\widetilde{\mathcal{A}}_n \supsetneq \mathcal{A}_n$ for $n \ge 5$.

\begin{corollary}
    The chamber decompositions of $\Delta_{n,2}$ induced by the $T^n$-actions on $G_{n,2}$ and $\mathbb{C}P^{N_2}$ coincide for $n = 4$ but differ for $n \ge 5$.
\end{corollary}

\begin{example}
    For $n \ge 5$, $\widetilde{\mathcal{A}}_n \setminus \mathcal{A}_n$ contains hyperplanes of the form $x_i = x_j$ for $1 \le i < j \le n$. The complete arrangement of all such hyperplanes is known as the \emph{braid arrangement} or the \emph{Coxeter arrangement of type $A$}.
\end{example}

In the following theorem, we determine the explicit form of the polytopes $P_\sigma$ corresponding to the hyperplanes in $\widetilde{\mathcal{A}}_n \setminus \mathcal{A}_n$.

\begin{theorem}
    Let $P_\sigma$ be an $(n-2)$-dimensional polytope given by the graph $G_\sigma = K_{p,q} \sqcup K_l$ with $p,q \ge 1$ and $l \ge 3$, i.e., $P_\sigma$ is non-admissible, intersects the interior of $\Delta_{n,2}$, and is not contained in any larger $(n-2)$-dimensional polytope. Then $P_\sigma$ is combinatorially equivalent to the join $(\Delta^{p-1} \times \Delta^{q-1}) * \Delta_{l,2}$.
\end{theorem}
\begin{proof}
Recall that the edges of $G_\sigma$ correspond to the vertices of $P_\sigma$. Moreover, two vertices of $P_\sigma$ are incident if the corresponding edges of $G_\sigma$ share a common vertex. Thus, we obtain the polytopes 
$$P(K_{p,q}) = \operatorname{conv}\{ \bm e^i + \bm e^j \mid (i,j) \in E(K_{p,q})\} \quad \text{and} \quad P(K_l) = \operatorname{conv}\{ \bm e^i + \bm e^j \mid (i,j) \in E(K_l)\},$$
which have no common vertices. Consequently, $P_{\sigma} = P(K_{p,q}) * P(K_l)$. But $P(K_{p,q}) \cong \Delta^{p-1} \times \Delta^{q-1}$ and $P(K_l) \cong \Delta_{l,2}$. Thus,
$$P_\sigma \cong (\Delta^{p-1} \times \Delta^{q-1}) * \Delta_{l,2}.$$
\end{proof}

We now describe all polytopes $P_\sigma$ of dimension $\dim P_\sigma \le n-2$. This follows immediately.

\begin{theorem}
    Let $P_\sigma$ be a polytope of dimension $\dim P_\sigma \le n-2$. Then $P_\sigma$ is the join $A * B$ of subpolytopes $A \subseteq \Delta^{p-1} \times \Delta^{q-1}$ and $B \subseteq \Delta_{l,2}$ with $p+q \ge 1$, $l \ge 0$, and $p+q+l = n$.
\end{theorem}

Polytopes of maximal dimension, namely dimension $n-1$, have a more complicated structure. They can be described as all possible nonempty intersections of half-spaces determined by hyperplanes from $\widetilde{\mathcal{A}}_n$ with the hypersimplex $\Delta_{n,2}$.

In particular, the maximal nonempty intersections yield chambers of maximal dimension in $\Delta_{n,2}$. Thus, we have proved the following.

\begin{theorem}
    The regions of the hyperplane arrangement $\widetilde{\mathcal{A}}_n$ are exactly the maximal-dimensional chambers in $\Delta_{n,2}$ for the $T^n$-action on $\mathbb{C}P^{N_2}$.
\end{theorem}

\begin{remark}
    Unlike the $T^n$-action on $G_{n,2}$, there exist polytopes $P_\sigma$ of dimension $\dim P_\sigma \le n-3$ that intersect the interior of the hypersimplex $\Delta_{n,2}$. Therefore, the hyperplane arrangement $\widetilde{\mathcal{A}}_n$ may not suffice to describe chambers of dimension less than $n-1$.
\end{remark}

According to \cite{BT2}, the regular values of the moment map $\widetilde \mu$ coincide with the chambers of maximal dimension. Thus, we have obtained

\begin{corollary}
    The regular values of the moment map $\widetilde \mu \colon \mathbb{C}P^{N_2} \to \Delta_{n,2}$ are the interior points $x \in \Delta_{n,2}$ such that $x \notin H$ for all hyperplanes $H \in \widetilde{\mathcal{A}}_n$.
\end{corollary}


\section{Concluding Remarks}

In this final section, we provide interpretations of the results on the $T^n$-action on $G_{n,2}$ obtained in the preceding sections. The \emph{complex of admissible polytopes} $C(G_{n,2}, \Delta_{n,2})$ is defined in \cite{BT1} as the pair 
$$C(G_{n,2}, \Delta_{n,2}) = \Big( \bigsqcup_{\sigma \text{ admissible}} \mathring{P}_\sigma,\; d_{C} \Big),$$
where the first component is the formal disjoint union of the relative interiors of $P_\sigma$, and $d_{C}$ is a boundary operator on this union that sends each $\mathring{P}_\sigma$ to the disjoint union of all its faces. The sets $\mathring{P}_\sigma$, for admissible $\sigma$, are called the \emph{cells} of the complex $C(G_{n,2}, \Delta_{n,2})$.

\begin{remark}
In Theorem~\ref{classification}, we showed that there is a bijection between the cells of the complex $C(G_{n,2}, \Delta_{n,2})$ and complete $N$-partite graphs on $m \le n$ vertices, where $N \ge 2$. The formula 
$$\dim \mathring{P}_\sigma = 
\begin{cases} 
m - 2, & \text{if } N = 2,\\[4pt]
m - 1, & \text{if } N \ge 3
\end{cases}$$
(Theorem~\ref{dimensionadmissiblepolytope}) provides a direct computation of the cell dimensions in terms of these graph invariants.
\end{remark}

For the complex $C(G_{n,2}, \Delta_{n,2})$ there is a natural projection
\begin{equation}\label{complexprojection}
    C(G_{n,2}, \Delta_{n,2}) \longrightarrow \Delta_{n,2},
\end{equation}
which maps each component $\mathring{P}_\sigma$ in the disjoint union to $\mathring{P}_\sigma$ as a subset of $\Delta_{n,2}$. Thus, when we restrict the projection to the components of $C(G_{n,2}, \Delta_{n,2})$, we obtain homeomorphisms onto their images. Nevertheless, for $n \ge 4$, the positive complexity of the $T^n$-action on $G_{n,2}$ implies that the projection (\ref{complexprojection}) is not a bijection.

\begin{remark}
    The operations of vertex isolation and merging of parts described in Theorem~\ref{(n-1)-dim_adm} provide an explicit combinatorial rule for computing the boundary operator $d_{C}$ in $C(G_{n,2}, \Delta_{n,2})$.
\end{remark}

\newpage


\subsection*{Acknowledgements}
The author expresses his deep gratitude to his supervisor, Victor M. Buchstaber, for fostering a creative atmosphere, providing generous guidance, and making an invaluable contribution to his mathematical education. The author thanks Svetlana Terzi\'c for kind and helpful discussions of the results of this paper.

The author is also grateful to Nikolai Yu. Erokhovets for drawing his attention to the embedding of the hypersimplex into the simplex, and to Taras E. Panov for drawing his attention to the GKZ decomposition.

The author thanks Eugene S. Kogan, Semyon O. Malygin and Fedor E. Vylegzhanin for many useful discussions and insightful suggestions.

The study has been funded within the framework of the HSE University Basic Research Program (HSE-BR-2025-023)

\normalsize
\end{document}